\documentclass[12pt,a4paper]{amsart}
\usepackage{amssymb, rotating,float}

\usepackage{tikz}
\newtheorem{theorem}{Theorem}[section]

\newtheorem{question}[theorem]{Question}

\newtheorem{lemma}[theorem]{Lemma}
\newtheorem{corollary}[theorem]{Corollary}
\newtheorem{definition}[theorem]{Definition}
\newtheorem{constr}[theorem]{Construction}

\theoremstyle{definition}
\newtheorem{remark}[theorem]{Remark}

\newcommand{\GL}{\mathrm{GL}}
\newcommand{\diag}{\mathrm{diag}}

\newcommand{\C}{{\mathcal C}}
\newcommand{\Inc}{\mathrm{Inc}}

\newcommand{\soc}{\mathrm{Soc}}

\mathchardef\mhyphen="2D
\newcommand{\nona}{\mathrm{non\mhyphen}\alpha}
\newcommand{\nonb}{\mathrm{non\mhyphen}\beta}
\newcommand{\nong}{\mathrm{non\mhyphen}\gamma}
\newcommand{\non}{\mathrm{non\mhyphen}}
\newcommand{\Rep}{\mathrm{Rep}}
\newcommand{\RS}{\mathrm{RepSet}}

\newcommand{\nl}{\par \bigskip \noindent}

\DeclareMathOperator{\Aut}{Aut}

\DeclareMathOperator{\Sym}{Sym}
\newcommand{\PG}{\mathop{\mathrm{PG}}}

\newcommand{\PGammaL}{\mathop{\mathrm{P}\Gamma\mathrm{L}}}

\newcommand{\AGL}{\mathop{\mathrm{AGL}}}

\newcommand{\GF}{\mathop{\mathrm{GF}}}

\DeclareMathOperator{\Wr}{wr}

\newcommand{\ourgeoms}[2]{${#1}^{{#2}}$-uniform geometry}

\def\dotcup{\DOTSB\mathop{\overset{\textstyle.}\cup}}
\def\bigdotcup{\DOTSB\mathop{\overset{\textstyle.}\bigcup}}

\title[Basic coset geometries]{Basic coset geometries}
\author[Michael Giudici, Geoffrey Pearce and Cheryl E.~Praeger]{Michael Giudici, Geoffrey Pearce and Cheryl E.~Praeger\\ \\Centre for the Mathematics of Symmetry and Computation\\ 
School of Mathematics and Statistics\\
University of Western Australia\\
Crawley WA 6009 Australia \\ \\
Michael.Giudici@uwa.edu.au, geoffrey.pearce@graduate.uwa.edu.au, Cheryl.Praeger@uwa.edu.au}
\date{}

\begin{document}

\begin{abstract}
In earlier work we gave a characterisation of pregeometries which are `basic' (that is, admit no `non-degenerate' quotients) relative to two different kinds of quotient operations, namely imprimitive quotients and normal quotients. Each basic geometry was shown to involve a faithful  group action, which is  primitive or quasiprimitive respectively, on the set of elements of each type.  
For each O'Nan-Scott type of primitive group, we construct a new infinite family of geometries, which are thick and of unbounded rank, and which admit a flag-transitive automorphism group acting faithfully  on the set of elements of each type as a primitive group of the given O'Nan-Scott type.
\end{abstract}

\maketitle

\section{Introduction} \label{intro}
The technique of taking quotients has proved useful for studying various classes of combinatorial objects.  For instance, in graph theory it has been successful in characterising distance-transitive graphs \cite{SmithDH}, $s$-arc transitive graphs \cite{PraegerQP} and locally $s$-arc transitive graphs \cite{GLP3}.  Both \cite{CGP} and \cite{firstpaper} represent a recent effort to develop a framework for studying geometries and pregeometries using this technique.  In \cite{firstpaper} we gave a characterisation of pregeometries which are `basic' (that is, admit no `non-degenerate' quotients) relative to one of two different kinds of quotient operations, namely imprimitive quotients and normal quotients.  The purpose of this paper is to demonstrate by construction that basic pregeometries can have arbitrarily large rank. Furthermore, the examples we construct satisfy a number of restrictive geometric   conditions which are important in the field of incidence geometry; namely, they are geometries, flag-transitive, thick and  connected (in the sense of having connected rank $2$ truncations). See Subsection \ref{incgeoms} for definitions of the geometrical terminology.

Each basic pregeometry arises from a basic pregeometry involving a group of automorphisms which is faithful and primitive or quasiprimitive, according to the kind of quotient,  on the set of elements of each type. Such groups are categorised as having one of several `O'Nan-Scott types' (see Subsection~\ref{QPgroups}) and our constructions cover each O'Nan-Scott type. To set our main results Theorems~\ref{maintheorem} and \ref{maintheorem2} in context we give a brief pr\'ecis of our approach to studying incidence geometries and the questions that have arisen.

\subsection{Primitive and quasiprimitive groups and geometry quotients}

Geometries which are flag-transitive belong to the class of coset pregeometries (see Section \ref{classC}), and their structure can be defined in terms of intersections of cosets within a group. Constructing geometries via cosets goes back to the work of Tits \cite{Tits} and there is an extensive body of literature on constructing coset geometries from various almost simple groups (for example \cite{bueetal,deho94,psl2,Leemans,leemanssuz,RS}). As demonstrated in \cite[Example 6.8]{CGP}, the quotient of a flag-transitive geometry may be neither flag-transitive nor a geometry (as opposed to a {\em pre}geometry), and this ultimately precludes a self-contained quotient theory of flag-transitive geometries.  However, by relaxing these two conditions we end up with the class of {\em coset pregeometries with connected rank $2$ truncations}, which {\em is} closed under taking both normal and imprimitive quotients, and which is therefore more amenable to study by this process. We are then faced with the problem of describing the basic pregeometries, that is, those with no meaningful quotients.

The two main results of \cite{firstpaper} deal with characterising these basic pregeometries, in the context of imprimitive quotients in \cite[Theorem 1.1]{firstpaper} and normal quotients in \cite[Theorem 1.2]{firstpaper}. An \emph{imprimitive quotient} is a quotient with respect to a  partition of the set of elements of the geometry that is invariant under a group $G$ of automorphisms. If the partition is also the set of orbits of some normal subgroup $N$ of $G$ then it is called a \emph{normal quotient}.
The basic  coset pregeometries for a group $G$ relative to these two kinds of quotients are called \emph{$G$-primitive basic} and \emph{$G$-normal-basic}, and involve faithful primitive and quasiprimitive group actions respectively: a permutation group is {\em primitive} if it leaves no proper, non-trivial partition of the point set invariant, and is {\em quasiprimitive} if every non-trivial normal subgroup is transitive; the latter is a generalisation of the former.  In essence the two theorems in \cite{firstpaper} state the following for $G$-vertex-transitive pregeometries (here $X_i$ denotes the set of all elements of type $i$ -- see Section \ref{prelim}):

\begin{itemize}
\item[(1)] {\em The study of $G$-primitive-basic pregeometries is reduced to studying those pregeometries in which $G$ is faithful and primitive on each $X_i$}, \cite[Theorem 1.1]{firstpaper}.

\item[(2)] {\em The study of $G$-normal-basic pregeometries is reduced to studying those pregeometries in which $G$ is faithful on every $X_i$ and quasiprimitive on all but at most one of the $X_i$,} \cite[Theorem 1.2]{firstpaper}.
\end{itemize}

Following on from this reduction, we consider in this paper the question of whether the rank of a flag-transitive pregeometry satisfying these conditions is bounded.  The well-known example of a projective space shows that this is not the case in general, and furthermore is an example of a geometry which is thick and flag-transitive, and which can have arbitrarily large rank.  Hence even with these restrictions the rank is unbounded.  

Thus we refine the question according to the different O'Nan-Scott types of primitive or quasiprimitive permutation groups.  The different types of primitive groups are described in \cite{praegerbcc}, in which each type is represented by a $2$ letter abbreviation (see Section \ref{QPgroups} for an explanation of these).  Moreover, the third author showed in \cite{PraegerQP} that quasiprimitive groups admit a similar characterisation with a direct correspondence between the different types of quasiprimitive and primitive groups (in the sense that a primitive group of a given type is a quasiprimitive group of the same type).  Thus, when constructing examples of geometries satisfying the conditions in (1) preserved by a given O'Nan-Scott type of primitive group, we are also constructing examples for that type of quasiprimitive group satisfying the conditions in (2).  This leads us to ask the question as follows.
\begin{center} {\em
Given a particular O'Nan-Scott type of primitive group, is there an upper bound for the rank of any thick flag-transitive geometry preserved by a group of this type?}
\end{center}
The main result of this paper is that the answer is `no', and the answer is still `no' in the particular case of geometries where the actions of $G$ on the $X_i$ are permutationally isomorphic.  Let $\Gamma$ be a geometry and $G$ a permutation group on a set $\Omega$. We say that $\Gamma$ is a \emph{\ourgeoms{G}{\Omega}} if $\Gamma$ is a thick geometry with connected rank 2 truncations, $G$ is flag-transitive on $\Gamma$, and for any type in $\Gamma$ the action of $G$ on the subset of elements of that type is permutationally isomorphic to the $G$-action on $\Omega$ (and in particular is faithful).

\begin{theorem} \label{maintheorem}
Let $k$ be a positive integer.  Then for each of the eight O'Nan-Scott types of primitive permutation groups, there exists a primitive group $G$ of that type on a set $\Omega$ and a \ourgeoms{G}{\Omega} of rank $k$.
\end{theorem}

Theorem \ref{maintheorem} is proved in Section \ref{mainthmproof}. We also show that for a given O'Nan-Scott type there is often a large amount of flexibility in how we can achieve a \ourgeoms{G}{\Omega} of rank $k$. In particular, $G$ has a minimal normal subgroup of the form $T^n$ for some simple group $T$ and integer $n$, and we can often achieve arbitrarily large rank by varying either $T$ or $n$.

Although in this paper we focus on constructing geometries in which the primitive action is of the same O'Nan-Scott type on each set $X_i$ it is possible for a group to have faithful primitive actions of more than one  O'Nan-Scott type. Indeed there exist a group $G$ which is flag-transitive on a thick geometry of rank 3 satisfying the conditions in (1) such that, for $i\neq j$, $G^{X_i}$ and $G^{X_j}$ are primitive of different O'Nan-Scott types. Such geometries are discussed briefly in Section \ref{differenttypes}.

As part of the paper we give several constructions of flag-transitive geometries that, to our knowledge, are new. For the geometries in Construction \ref{increasing1s}, we determine the basic diagram in Theorem \ref{thm:proddiag}; for the geometries in Constructions \ref{HSconstruction} and \ref{SDincreasing1s} we observe that the automorphism groups are different from the automorphism groups of any other family of geometries of unbounded rank known to us (see Remarks \ref{rem:HSdiag} and \ref{rem:SDdiag}). A key ingredient of many of these constructions is a generic construction (see Section \ref{genericconstr}) for coset pregeometries with connected rank 2 truncations which is analogous to the construction of orbital graphs in graph theory (just as coset pregeometries are analogous to coset graphs).  

\begin{theorem} \label{maintheorem2}
Let $\Gamma$ be a coset pregeometry of rank $k$ with connected rank $2$ truncations, and let $i$ be a type in $\Gamma$.  Then $\Gamma$ arises by applying Construction {\em \ref{balloon}} recursively $k-1$ times starting with a rank $1$ pregeometry whose elements form the set of type $i$ elements of $\Gamma$.
\end{theorem}

In addition, we identify the conditions under which the generic construction yields a flag-transitive geometry.  This generic construction forms the basis of the constructions in Sections \ref{PAtype} to \ref{SDtype} which are used to build many of the examples of basic geometries needed for the proof of Theorem \ref{maintheorem} in Section \ref{mainthmproof}.

\subsection*{Acknowledgements}
The authors thank an anonymous referee for suggestions that improved the paper and Philippe Cara for helpful discussions.

\section{Preliminaries} \label{prelim}
\subsection{Pregeometries} \label{incgeoms}
A {\em pregeometry} $\Gamma = (X,\ast,t)$ consists of a set $X$ of {\em elements} (often called {\em points}) with an {\em incidence relation} $\ast$ on the points, and a map $t$ from $X$ onto a set $I$ of {\em types}.  The incidence relation is symmetric and reflexive, and if $x \ast y$ we say that $x$ and $y$ are {\em incident}.  Furthermore if $x \ast y$ with $x \neq y$ then $t(x) \neq t(y)$.  For each $i \in I$, we use the notation $X_i$ to mean the set $t^{-1}(i)$ of all elements of type $i$.  It follows that $X$ is the disjoint union $\bigcup_{i \in I} X_i$. The number $|I|$ of types is called the {\em rank} of the pregeometry, and we assume throughout the paper that $|X|$, and hence $|I|$, is finite.  Unless stated otherwise, the set $I$ for a rank $k$ pregeometry is equal to $\{1,\ldots,k\}$. 
A {\em flag} $F$ is a set of pairwise incident elements of $\Gamma$ (which implies that the elements of $F$ are of pairwise distinct types). The \emph{rank} of a flag is the number of elements that it contains.  A {\em chamber} is a flag containing one element of each type.  A pregeometry in which every flag is contained in a chamber is called a {\em geometry}, and a geometry in which each non-maximal flag is contained in at least two (respectively three) chambers is called {\em firm} (respectively {\em thick}).  

For a nonempty subset $J \subseteq I$, the \emph{$J$-truncation} of $\Gamma$ is the pregeometry $(X_J,\ast_J,t_J)$ where $X_J = t^{-1}(J)$, $\ast_J$ is the restriction of $\ast$ to $X_J$, and $t_J$ is the restriction of $t$ to $X_J$. 

The \emph{incidence graph} of a pregeometry  $\Gamma = (X,\ast,t)$ is the graph with vertex set $X$ and two elements $x,y\in X$ are adjacent if $x*y$ and $x\neq y$.
A pregeometry $\Gamma$ is said to be \emph{connected} if for any two elements $x,y\in X$ there is a sequence $x=x_1,x_2\ldots,x_{k-1},x_k=y$ with $x_i\ast x_{i+1}$, that is, the incidence graph is connected. We say that $\Gamma$ has \emph{connected rank $2$ truncations} if for each $i,j\in I$ with $i\neq j$, the $\{i,j\}$-truncation is connected.

Let $\Gamma=(X,\ast,t)$ be a pregeometry and let $\mathcal{B}$ be a partition of $X$ such that for each $B\in\mathcal{B}$, all elements of $B$ have the same type. The \emph{quotient pregeometry} of $\Gamma$ with respect to $\mathcal{B}$ is the pregeometry whose elements are the parts of $\mathcal{B}$ and two parts $B_1,B_2$ are incident if there exists $x_1\in B_1$ and $x_2\in B_2$ such that $x_1*x_2$ in $\Gamma$. The type function is inherited from $\Gamma$.

Let $\Gamma = (X,\ast,t)$ be a pregeometry.  An {\em automorphism} of $\Gamma$ is a permutation $g$ of $X$ such that $t(x) = t(x^g)$ for all $x \in X$, and $x^g \ast y^g$ if and only if $x \ast y$, for all $x, y \in X$.  We write $\Aut\Gamma$ for the automorphism group of $\Gamma$.  

Let $G \leq \Aut\Gamma$. For $i \in I$ we write $G_{(X_i)}$ for the kernel of the action of $G$ on $X_i$, and $G^{X_i}$ for the group induced by $G$ on $X_i$ (isomorphic to $G/G_{(X_i)}$).  For $x_i \in X_i$, we write $x_i^G$ for the orbit of $x_i$ under $G$.

For a subset $J$ of the type set $I$, we say that $G$ is \emph{$J$-flag-transitive} on $\Gamma$ if $G$ acts 
transitively on the set of all flags $F$ with $t(F) = J$. We say that $G$ is \emph{vertex-transitive} on $\Gamma$ if $G$ is $J$-flag-transitive on $\Gamma$ for all $J$ with $|J|=1$, \emph{incidence-transitive} on $\Gamma$ if $G$ is $J$-flag-transitive for all $J$ with $|J|=2$, and \emph{chamber-transitive} on $\Gamma$ if $G$ is $I$-flag-transitive. If $G$ is $J$-flag-transitive for all $J\subseteq I$ we say that $G$ is \emph{flag-transitive} on $\Gamma$.  If $\Aut(\Gamma)$ is flag-transitive on $\Gamma$ we often simply say that  $\Gamma$ is \emph{flag-transitive}. If $\Gamma$ is a geometry then $G$ is chamber-transitive on $\Gamma$ if and only if $G$ is flag-transitive on $\Gamma$.

Let $F$ be a flag of a pregeometry $\Gamma=(X,\ast,t)$. The \emph{residue} of $F$ in $\Gamma$, denoted $\Gamma_F$, is the pregeometry $(X_F,\ast_F,t_F)$ induced by $\Gamma$ on the set $X_F$ of all elements of $X$ incident with each element of $F$ and whose type is not in $t(F)$. If $\Gamma$ is a geometry then so is the residue $\Gamma_F$. The rank two residues are $\Gamma$ are the residues of flags of rank $|I|-2$.

Given a geometry $\Gamma$ with type set $I$, the \emph{basic diagram} of $\Gamma$ is the graph with vertex $I$ such that two types $i,j$ are adjacent if and only if there is a rank two residue $\Gamma_F$ of type $\{i,j\}$ (that is, a residue of a flag $F$ of type $I\backslash\{i,j\}$), for which the graph induced on $X_F$ by the incidence relation is not complete bipartite. Note that if $G$ is flag-transitive  on $\Gamma$ then all residues of a given type are isomorphic. Sometimes we decorate the basic diagram by placing labels on each edge to describe the particular residue.

\subsection{Primitive and quasiprimitive permutation groups} \label{QPgroups}
Finite primitive permutation groups, and also finite quasiprimitive groups can be divided into a number of distinct classes according to their action and the structure of their socle; see for example \cite[Chapter 4]{DM} and \cite{praegerbcc} (the socle of a group is the subgroup generated by all its minimal normal subgroups).  We follow the class division given in \cite{praegerbcc}, which we outline in Table \ref{primgroups}. Table \ref{primgroups} also gives a defining condition enabling the O'Nan-Scott type to be identified (additional conditions are sometimes required for primitivity).

\begin{table}[ht*]
\begin{tabular}{cll}
\hline
Abbreviation & O'Nan-Scott type& Defining property\\
\hline
HS & Holomorph simple & \parbox[t]{2.5in}{Two minimal normal subgroups, each isomorphic to $T$}\\
& & \\
HC & Holomorph compound & \parbox[t]{2.5in}{Two minimal normal subgroups, each isomorphic to $T^n$ for some $n \geq~2$}\\
& & \\
HA & Affine (abelian holomorph) & \parbox[t]{2.5in}{Abelian minimal normal subgroup}\\
& & \\
AS & Almost simple & \parbox[t]{2.5in}{Unique minimal normal subgroup isomorphic to $T$}\\
& & \\
SD & Simple diagonal & \parbox[t]{2.5in}{Unique minimal normal subgroup isomorphic to $T^n$ which acts non-regularly and $(T^n)_\alpha \cong T$ with $n\geq2$}\\
& & \\
CD & Compound diagonal & \parbox[t]{2.5in}{Unique minimal normal subgroup isomorphic to $T^n$, and $(T^n)_\alpha$ isomorphic to $T^\ell$ for some $\ell \geq 2$}\\
& & \\
TW & Twisted wreath & \parbox[t]{2.5in}{Unique minimal normal subgroup isomorphic to $T^n$ which acts regularly, for some $n\geq2$}\\
& & \\
PA & Product action & \parbox[t]{2.5in}{Unique minimal normal subgroup isomorphic to $T^n$ for some $n\geq 2$, $(T^n)_{\alpha}\neq 1$ and $(T^n)_{\alpha}\not\cong T^\ell$ for any $\ell\geq 1$.}\\
\hline
\end{tabular}
\caption{The different kinds of primitive and quasiprimitive permutation groups on $\Omega$. Here $T$ is a nonabelian simple group and $\alpha\in\Omega$.} \label{primgroups}
\end{table}

\subsubsection{Primitive groups acting in product action} \label{QPPA}
Although the name `Product Action' is reserved for the PA type of groups described in Table \ref{primgroups}, primitive groups of O'Nan-Scott types HC, CD, and TW, and certain primitive groups of type HA, can also be viewed as groups acting in product action (differing from the PA groups in the structure of their socles and point stabilisers).

The product action of a wreath product $G = H \Wr S_n = H^n \rtimes S_n$ is the action on $\Omega=\Delta^n$ (where $H$ acts on $\Delta$) given by 
$$(\delta_1,\ldots,\delta_n)^{(h_1,\ldots,h_n) \sigma^{-1}} = (\delta_{1^{\sigma}}^{h_{1^\sigma}},\ldots,\delta_{n^{\sigma}}^{h_{n^\sigma}})$$
for all $(\delta_1,\ldots,\delta_n) \in \Omega$ and all $(h_1,\ldots,h_n) \sigma^{-1} \in G$.  Such a group is primitive when the group $H$, called the component of $G$, is primitive but not regular on $\Omega$, see \cite[Theorem 4.5]{CameronBook}.
The O'Nan-Scott types of $G$ corresponding to various O'Nan-Scott types of $H$ are listed in Table \ref{PAexpansions}.

\begin{table}
\begin{tabular}{|c|c|c|c|c|c|c|c|c|}
\hline
Type of $H$ & AS & SD & HS & HA & TW & CD & HC & PA\\
\hline
Type of $G$ & PA & CD & HC & HA & TW & CD & HC & PA\\
\hline
\end{tabular}
\caption{O'Nan-Scott types of primitive groups $G = H \Wr S_n$ according to the O'Nan-Scott type of the component $H$.} \label{PAexpansions}
\end{table}

This observation is important since it enables us to use a single construction (in Section \ref{PAtype}) to build geometries of arbitrary rank preserved by primitive groups of types PA, CD, HC, HA or TW, and this forms part of the proof of Theorem \ref{maintheorem}.

\section{A generic construction and proof of Theorem \ref{maintheorem2}} \label{genericconstr}

Here we give a generic construction for pregeometries lying in a certain family $\C$  which contains all the $G^\Omega$-uniform geometries (as defined in the introduction). First we define $\C$.

\subsection{The class $\C$ of pregeometries} \label{classC}
Given a group $G$ with a set of subgroups $\{G_i\}_{i \in I}$, the {\em coset pregeometry} is the pregeometry whose elements of type $i \in I$ are the right cosets of $G_i$ in $G$, such that two cosets $G_ix$ and $G_jy$ are incident if and only if $G_ix \cap G_jy$ is nonempty.  The coset pregeometry is denoted by $\Gamma(G,\{G_i\}_{i \in I})$.

Let $\C$ be the set of pregeometries $\Gamma = (X,\ast,t)$ which are isomorphic to coset pregeometries with connected rank $2$ truncations.  Then $\C$ is the set of all pregeometries such that all of the following hold \cite[Theorem 6.7]{CGP}.
\begin{itemize}
\item[(i)]  $\Aut\Gamma$ is vertex-transitive and incidence-transitive on $\Gamma$.
\item[(ii)] $\Gamma$ contains a chamber.
\item[(iii)] The rank $2$ truncations of $\Gamma$ are connected.
\end{itemize}
It is clear from the definition that the class $\C$ contains all $G^{\Omega}$-uniform geometries.
As explained in the introduction, we are interested in the set $\C$ because it is closed under taking normal and imprimitive quotients.  That is, given $\Gamma \in \C$ and a $G$-invariant partition ${\mathcal P}$ of $\Gamma$ such that for each $P\in\mathcal{P}$ all elements of $P$ are of the same type, the quotient $\Gamma_{/ \mathcal P}$ also lies in $\C$ \cite[Theorem 1.5]{CGP}.

\subsection{The construction}
Essentially, we show how to add a new type and corresponding set of elements to an existing pregeometry in $\C$ to create a new pregeometry whose rank is therefore one greater than the old one.  It is a generic construction because (as we show in Theorem \ref{reverseballoonthm}) any pregeometry in $\C$ can be built up starting from one set $X_1$ and then successively adding sets $X_i$ using the construction.
\begin{constr} \label{balloon}
{\em We list the input and output for the construction.

\nl
Input: $\Gamma' = (X',\ast',t')\in\mathcal{C}$ of rank $k'$ with set of types $I'$;\\
\indent $G \leq \Aut\Gamma'$\\
\indent $K' = \{x_i \, | \, i \in I'\}$, a chamber in $\Gamma'$ with $t'(x_i) = i$ for each $i \in I'$;\\
\indent a set $Y$ on which $G$ acts transitively, an element $y \in Y$, and a symbol $j \not\in I'$.

\nl
Output:
$\Inc(\Gamma',G,K',Y,y,j)$ with properties given in Lemma \ref{balloonlem}

\nl
Here $\Inc(\Gamma',G,K',Y,y,j)$ is the incidence structure $(X,\ast,t)$ with point set $X = X' \cup Y$, set of types $I' \cup \{j\}$, type function $t : x \longmapsto t'(x)$ if $x \in X'$, and $t : x \longmapsto j$ if $x \in Y$, and reflexive, symmetric incidence relation $\ast$ defined as follows.  For each $x_i \in K'$ we define incidence between points in $Y$ and in $(X')_i$ such that the points in $Y$ incident to $x_i$ are those in the orbit of $y$ under $G_{x_i}$, and the other incident pairs (between $Y$ and $(X')_i$) are the images under $G$ of those incident pairs containing $x_i$.  That is to say, we define $\ast$ on $X' \times Y$ by
$$x_i^g \ast w^g \; {\rm for \; all} \; \; w \in y^{G_{x_i}}, g \in G, {\rm and} \; x_i \in K'.$$
Otherwise, for $u$, $v$ both in $X'$, we set $u \ast v$ if and only if $u \ast' v$.
}
\end{constr}

\begin{lemma} \label{balloonlem}
Let $\Gamma'$, $G$, $K'$, $Y$, $y$, $k'$ and $j$ be as in Construction {\em \ref{balloon}}, and let $\Gamma = \Inc(\Gamma',G,K',Y,y,j)=(X,*,t)$.  Then the following all hold.
\begin{itemize}
\item[(a)] $\Gamma$ is a pregeometry of rank $k' + 1$, $K' \cup \{y\}$ is a chamber of $\Gamma$, $G \leq \Aut\Gamma$,  and $G$ is vertex-transitive and incidence-transitive on $\Gamma$.  
\item[(b)] The rank $2$ truncations of $\Gamma$ are connected (that is to say, $\Gamma$ is in $\C$) if and only if, for all $x_i \in K'$, $\langle G_{x_i}, G_y \rangle = G$.
\end{itemize}
\end{lemma}
\begin{proof}
It follows from the definition of $\Gamma$ that $\Gamma$ is a pregeometry of rank $k'+1$. Also the elements of $K' \cup \{y\}$ all have distinct types under $t$, and $y*y'$ for each $y'\in K'$, so $K' \cup \{y\}$ is a chamber of $\Gamma$.  That $G \leq \Aut\Gamma$ follows from the definition of $\ast$, and $G$ is transitive on each $X_i$ and acts incidence-transitively on $\Gamma$.  Let $i$ be a type not equal to $j$.  The incidence graph of the $\{i,j\}$-truncation has edge set equal to $\{x_i,y\}^G$, where $x_i$ is the type $i$ point in $K'$; hence it is connected if and only if $\langle G_{x_i}, G_y \rangle = G$ (see for example \cite[Lemma 3.7(1)]{GLP3}).
\end{proof}

Next we show that pregeometries satisfying properties (i) and (ii) of the definition of $\C$ in Subsection~\ref{classC} can be `decomposed' via Construction~\ref{balloon}.
In the proof we use the notation $\Gamma_x$, for $x \in X$, to denote the set $\{y \in X \,  | \, x \ast y, y\neq x\}$.

\begin{theorem} \label{reverseballoonthm}
Let $\Gamma = (X,\ast,t)$ be a rank $k$ pregeometry with type set $I$, and $G \leq \Aut\Gamma$ such that
\begin{itemize}
\item[(a)] $G$ is vertex-transitive and incidence-transitive on $\Gamma$, and
\item[(b)] $\Gamma$ contains a chamber $K$.
\end{itemize}
Let $j \in I$, let $y$ be the unique point of $K$ in $X_j$, and let $K' = K \backslash \{y\}$.  Let $\Gamma'$ be the $(I \backslash \{j\})$-truncation of $\Gamma$.  Then $\Gamma = \Inc(\Gamma',G,K',X_j,y,j)$ as defined in Construction {\em \ref{balloon}}.
\end{theorem}
\begin{proof}
We have to prove only that the incidence in $\Gamma$ is the same as that specified in Construction \ref{balloon}.   Let $x \in K'$.  Since $K$ is a chamber we know that $y \ast x$.  Also since $G$ is incidence-transitive on $\Gamma$, $\Gamma_x \cap X_j = y^{G_x}$.  The group $G$ is transitive on each $X_i$ and so for any $x'\in X$ such that $t(x') = t(x)$, there is $g \in G$ with $x^g = x'$.  Since $G$ preserves incidence it follows that $\Gamma_{x'} = \Gamma_x^g$; that is, $x' \ast (y')^g$ for all $g \in G$ and $y' \in y^{G_{x}}$.  Since each $x' \in X \backslash X_j$ arises for some $x \in K'$, it follows that $\Gamma = \Inc(\Gamma',G,K',X_j,y,j)$ as defined in Construction \ref{balloon}.
\end{proof}

Theorem \ref{maintheorem2} now follows as a corollary to Theorem \ref{reverseballoonthm}.

Next we give necessary and sufficient conditions under which $\Gamma=\Inc(\Gamma',G,K',Y,y,j)$ is a geometry and $G$ is flag-transitive on $\Gamma$. Part (b) is a reinterpretation of the condition for a coset pregeometry to be a geometry (see \cite{deho94} or \cite[p79]{IncGeomHB})) to our recursive construction.

\begin{lemma} \label{balloonchambtrans}
Suppose that $G$ is chamber-transitive on the pregeometry $\Gamma'$ in $\C$ with chamber $K'$, and suppose that $G$ acts transitively on a set $Y$.  Let $y \in Y$, let $j$ be a symbol not in the set of types for $\Gamma'$, and let $\Gamma = \Inc(\Gamma',G,K',Y,y,j)$ as in Construction {\em \ref{balloon}}.  Then
\begin{itemize}
\item[(a)] $G$ is chamber-transitive on $\Gamma$ if and only if  $G_{(K')}$ is transitive on $\bigcap_{x \in K'} y^{G_x} = \bigcap_{x \in K'} \Gamma_x$.
\item[(b)] $\Gamma$ is a geometry and $G$ is flag-transitive on $\Gamma$ if and only if $\Gamma'$ is a geometry and for every flag $Q$ of $\Gamma'$ contained in $K'$, $G_{(Q)}$ is transitive on $\bigcap_{x \in Q} y^{G_x} = \bigcap_{x \in Q} (\Gamma_x \cap Y)$.
\end{itemize}
\end{lemma}
\begin{proof}
First we prove part (a).  Assume that $G$ is chamber-transitive on $\Gamma$.  The set of chambers of $\Gamma$ containing $K'$ consists of all sets of the form $K' \cup \{y'\}$ for $y' \in \bigcap_{x \in K'} y^{G_x} \subseteq Y$.  Hence $G_{(K')}$ is transitive on $\bigcap_{x \in K'} y^{G_x}=\bigcap_{x \in K'} \Gamma_x$.

On the other hand, assume that $G_{(K')}$ is transitive on $\bigcap_{x \in K'} \Gamma_x$, let $L$ be a chamber of $\Gamma$, and let $K$ be the chamber $K' \cup \{y\}$ given by Construction \ref{balloon}.  Now $L$ contains a chamber $L'$ of $\Gamma'$ and by assumption there exists $g \in G$ with $(L')^g = K'$.  By Lemma \ref{balloonlem}, $G \leq \Aut\Gamma$, so the image of $L$ under $g$ is also a chamber of $\Gamma$, and the unique point $y'$ in $L^g \cap Y$ is contained in $\bigcap_{x \in K'} \Gamma_x$.  Hence there exists $h \in G_{(K')}$ mapping $y'$ to $y$, and so $L^{gh} = K$.  It follows that $G$ is chamber-transitive on $\Gamma$.

We now prove part (b).  Assume that $\Gamma$ is a geometry and that $G$ is flag-transitive on $\Gamma$. Let $F$ be a non-maximal flag of $\Gamma'$.  Then $F$ is contained in a chamber $L$ of $\Gamma$, and there exists a chamber $L'$ of $\Gamma'$ such that $F \subseteq L' \subset L$.  Hence $\Gamma'$ is a geometry.   Any subset $Q$ of $K'$ is a flag of $\Gamma$ and since $G$ is flag-transitive on $\Gamma$, $G_{(Q)}$ is transitive on the set of $(t(Q)\cup\{j\})$-flags of $\Gamma$ containing $Q$.  This implies that $G_{(Q)}$ is transitive on $\bigcap_{x \in Q} (\Gamma_x \cap Y)=\bigcap_{x \in Q} y^{G_x}$.

Conversely, assume that the latter condition in the statement holds.  Since this holds with $Q = K'$, part (a) implies that $G$ is chamber-transitive on $\Gamma$.  To prove that $\Gamma$ is a geometry (and hence also that $G$ is flag-transitive on $\Gamma$) let $F$ be a non-maximal flag of $\Gamma$.  If $F$ is contained entirely in $\Gamma'$, then $F$ is contained in a chamber $L'$ of $\Gamma'$ since $\Gamma'$ is a geometry, and since $G$ is chamber-transitive on $\Gamma'$ there exists $g \in G$ such that $(L')^g = K'$.  Hence $(K' \cup \{y\})^{g^{-1}}$ is a chamber of $\Gamma$ containing $F$.  On the other hand, if $F$ is not contained in $\Gamma'$, then $F = F' \cup \{y'\}$ for some flag $F'$ of $\Gamma'$ and point $y'$ in $Y$.  Now $F'$ is contained in a chamber $L'$ of $\Gamma'$ and hence there exists $h \in G$ with $(L')^h = K'$.  So $F^h = F'' \cup \{(y')^h\}$ where $F'' := (F')^h$ is a flag of $\Gamma'$ contained in $K'$.  By assumption there is an element $s \in G_{(F'')}$ mapping $(y')^h$ to $y$.  Hence $(K' \cup \{y\})^{s^{-1}h^{-1}}$ is a chamber of $\Gamma$ containing $F$, and it follows that $\Gamma$ is a geometry and $G$ is flag-transitive on $\Gamma$.
\end{proof}

We now give a demonstration of how Construction \ref{balloon} can be used to construct a \ourgeoms{G}{\Omega} starting with a group $G$ acting on a set $\Omega$ and a subset $\Sigma$ of $\Omega$. We first construct a geometry $(X_1,\ast,t)$ that has only one type, and that has $X_1=\Omega$ admitting the given $G$-action so that $K_1=\{x_1\}$ is a chamber, where $x_1 \in \Sigma$.  For each of the remaining points $x_i \in \Sigma$, we apply an iteration of Construction \ref{balloon}, adding a new type $i$, a new set $X_i$ (a copy of $\Omega$) and a new point $x_i\in X_i$ to the previous chamber.  We then prove that the result is a $G^{\Omega}$-uniform geometry (this is dependent on having chosen $\Sigma$ appropriately).

\begin{constr} \label{ASconstruction}
{\em
Let $m$ be an integer, and let $G = S_m$ acting on $\Omega = \{1,2,\ldots,m\}$.  For $i = 1$ to $m$ let $x_i = i \in \Omega$, and let $\Sigma = \{x_i \, | \, 1 \leq i \leq m-2\}$.  Let $X_1 = \Omega$, let $\Gamma_1 = (X_1,\ast,t)$, where $\ast$ consists of the pairs $(x,x)$ for $x \in X_1$, and $t(x) = 1$ for all $x \in X_1$.  Let $K_1 = \{x_1\}$.  For $b$ with $1 < b \leq m-2$, suppose we have constructed $\Gamma_{b-1}$ and chamber $K_{b-1} = \{x_1, \ldots, x_{b-1}\}$.  Let $\Gamma_b = \Inc(\Gamma_{b-1},G,K_{b-1},\Omega,x_b,b)$ as in Construction \ref{balloon}, and $K_b = K_{b-1} \cup \{x_b\}$.
}
\end{constr}

\begin{lemma} \label{ASconstructionlem}
For $1 \leq b \leq m-2$, let $\Gamma_b$ be the pregeometry given by Construction {\em \ref{ASconstruction}}. Then
\begin{itemize}
\item[(a)] $\Gamma_b$ is a geometry of rank $b$ and $G$ is flag-transitive on $\Gamma_b$,
\item[(b)] $\Gamma_b$ is thick if $b > 1$, and
\item[(c)] if $b > 1$ then the rank $2$ truncations of $\Gamma_b$ are connected (and hence $\Gamma_b$ is a \ourgeoms{G}{\Omega}).
\end{itemize}
\end{lemma}
\begin{proof}
Part (a): The proof is by induction on $b$.  As chambers are singletons for $\Gamma_1$, the statement is trivially true for $\Gamma_1$.  Suppose inductively that $1 < b \leq m-2$ and that $\Gamma_{b-1}$ is a geometry with $G$ flag-transitive on $\Gamma_{b-1}$.  We check that the condition in Lemma \ref{balloonchambtrans}(b) holds for $\Gamma_b = \Inc(\Gamma_{b-1},G,K_{b-1},\Omega,x_{b},b)$.  Let $Q$ be a non-empty subset of $K_{b-1}$.  We need to show that $\bigcap_{x_s \in Q} G_{x_s}$ is transitive on $\bigcap_{x_s \in Q} x_b^{G_{x_s}}$.

Let $x_s \in Q$.  Since $G = S_m$, $G_{x_s}$ is transitive on $\Omega \backslash \{x_s\}$.  Hence $\bigcap_{x_s \in Q} x_b^{G_{x_s}} = \bigcap_{x_s \in Q} \Omega \backslash \{x_s\} = \Omega \backslash Q$.  Now $\bigcap_{x_s \in Q} G_{x_s} = G_{(Q)} = \Sym(\Omega \backslash Q)$ which is transitive on $\Omega \backslash Q$.  Hence the condition in Lemma \ref{balloonchambtrans}(b) holds and so $\Gamma_{b}$ is a geometry and  $G$ is flag-transitive on $\Gamma_b$.

\nl
Part (b): Consider the chamber $K_b = \{x_1,\ldots,x_b\}$ with $b \geq 2$, and let $K'$ be any co-rank $1$ flag contained in $K_b$, with $x_i$ the unique point in $K_b \backslash K'$.  Then (as shown in the previous paragraph) $G_{(K')} = \Sym(\Omega \backslash K')$, and since $b \leq m-2$, $|\Omega \backslash K'| \geq 3$.  Since $G_{(K')}$ is transitive on $\Omega\backslash K'$, $K'$ is contained in $|\Omega\backslash K'|\geq 3$ chambers, and $\Gamma_b$ is a thick geometry.

\nl
Part (c): Since $G$ is primitive, $\langle G_{x_i}, G_{x_j} \rangle = G$ for distinct $x_i, x_j \in \Sigma$.  Hence by Lemma \ref{balloonlem}(b), the rank $2$ truncations are connected
\end{proof}

\begin{remark}
 We note that the geometry $\Gamma_b$, in Construction \ref{ASconstruction} is the well known geometry  constructed as follows: $\Gamma_b = (X,\ast, t)$ with type set $I = \{1,\ldots,b\}$, and $X = \bigdotcup_{i \in I} X_i$ where each $X_i$ is a copy of $\Omega$.  For each $i \in I$ there is a bijection $f_i : X_i \longmapsto \Omega$, and  incidence is defined as follows: for $x,y\in X$, $x \ast y$ if and only if either $x = y$, or $t(x) \neq t(y)$ and $f_{t(x)} (x) \neq f_{t(y)} (y)$.
\end{remark}

\section{Two product action constructions} \label{PAtype}

In this section we give two constructions which produce geometries where the action of $G$ on each set of elements of a given type is a product action. The first starts with a flag-transitive geometry of rank $k$ and set of elements $X=X_1\dotcup X_2\dotcup\ldots\dotcup X_k$ and for each positive integer $n$ constructs a new flag-transitive geometry of rank $k$ with set of elements $X_1^n\dotcup X_2^n\dotcup\ldots\dotcup X_k^n$. The second starts with a primitive group $G$ acting with a product action on $\Omega=\Delta^n$ and for all $k\leq \lfloor n/2\rfloor-1$ produces a \ourgeoms{G}{\Omega} of rank $k$.

\subsection{Forming the product of a geometry}

Given a geometry with set of elements $X=X_1\dotcup X_2\dotcup\ldots\dotcup X_k$ and a positive integer $n$ we construct a new geometry with set of elements $X_1^n\dotcup X_2^n\dotcup\ldots\dotcup X_k^n$. 

\begin{constr} 
\label{con:prod1}
Let $\Gamma=(X,*,t)$ be a rank $k$ geometry and for each $i\in I$ let $X_i=t^{-1}(i)$. For a positive integer $n$, let $\Gamma^n=(X',*',t')$ be the rank $k$ pregeometry whose set of elements is $X'=X_1^n\dotcup\ldots \dotcup X_k^n$ equipped with the map $t':X'\rightarrow I$ such that $t'(x)=i$ for each $x\in X_i^n$. Moreover, $(x_{i1},\ldots,x_{in})*'(y_{j1},\ldots,y_{jn})$  with $x_{i\ell}
\in X_i$ and $y_{j\ell}\in X_j$, if and only if $x_{i\ell}*y_{j\ell}$ for each $\ell=1,\ldots, n$.
\end{constr}

\begin{lemma}\label{lem:productup}
Let $\Gamma$ be a geometry, $n$ a positive integer, and let $\Gamma^n$ be the pregeometry yielded by Construction {\rm\ref{con:prod1}}. Then
\begin{enumerate}
 \item $\Gamma^n$ is a geometry.
\item If $\Gamma$ is firm (respectively thick) then $\Gamma^n$ is firm (respectively thick).
\item If each rank $2$ truncation of $\Gamma$ is connected then each rank $2$ truncation of $\Gamma^n$ is connected.
\item If $G$ is flag-transitive on $\Gamma$ then $G\Wr S_n$ is flag-transitive on $\Gamma^n$.
\item If $\Sigma$ is the residue of a flag $F$ of type $J$ in $\Gamma$ then $\Sigma^n$ is a residue of a flag of type $J$ in $\Gamma^n$. Moreover, $\Gamma^n$ has the same basic  diagram as $\Gamma$.
\end{enumerate}
\end{lemma}
\begin{proof}
Let $J\subseteq I$ and $F=\{(x_{i1},\ldots,x_{in})\mid i\in J\}$ be a flag in $\Gamma^n$ of type $J$. Then for each $\ell$, $\{x_{i\ell}\mid i\in J\}$ is a flag in $\Gamma$ and so extends to a chamber $\{x_{i\ell}\mid i\in I\}$. Thus $\{(x_{i1},\ldots,x_{in})\mid i\in I\}$ is a chamber of $\Gamma^n$ and so  $\Gamma^n$ is a geometry. Moreover, if each flag of $\Gamma$ is contained in at least $r$ chambers then each flag of $\Gamma^n$ is contained in at least $r^n$ chambers. Hence (2) follows.

Let $(x_{i1},\ldots,x_{in}),(y_{j1},\ldots,y_{jn})\in\Gamma^n$ of type $i$ and $j$ respectively with $i\neq j$. If each rank two truncation of $\Gamma$ is connected, then we can find a path in the $\{i,j\}$-truncation of $\Gamma$ between $x_{i\ell}$ and $y_{j\ell}$ for each $\ell$. Moreover, each such path has odd length, so if we have a path of length $r_\ell$ from $x_{i\ell}$ to $y_{j\ell}$ we can also find one of length $r_\ell+2p$ for all positive integers $p$. This allows us to construct a path of length $\max\{r_{\ell}:\ell\in \{1,\ldots,n\}\}$ in $\Gamma^n$ from  $(x_{i1},\ldots,x_{in})$ to $(y_{j1},\ldots,y_{jn})$. Thus all rank two truncations of $\Gamma^n$ are connected.

Suppose now that $G$ is flag-transitive on $\Gamma$ and let $F_1=\{(y_{i1},\ldots,y_{in})\mid i\in J\}$ be another flag of $\Gamma^n$ of type $J$. Then for each $\ell\in \{1,\ldots,n\}$ there exists $g_\ell\in G$ such that $\{x_{i\ell}\mid i\in J\}^{g_\ell}=\{y_{i\ell}\mid i\in J\}$. Hence $(F)^{(g_1,\ldots,g_n)}=F_1$ and so both $G^n$ and $G\Wr S_n$ are flag-transitive on $\Gamma^n$. 

Let $F=\{y_i\mid i\in J\}$ be a flag of type $J$ in $\Gamma$. Then $F'=\{(y_i,\ldots,y_i)\mid i\in J\}$ is a flag of type $J$ in $\Gamma^n$. Moreover, for $r\notin J$ the element $(x_{r1},\ldots,x_{rn})$ is in the residue of $F'$ if and only if each  for each $k\leq n$, $x_{rk}$ is in the residue of the flag $\{y_{ik}\}_{i\in J}$ of $\Gamma$. In particular, if $\Sigma$ is the residue of $F$ then the residue of $F'$ is $\Sigma^n$. Moreover, if $F''=\{(y_{i1},\ldots,y_{in})\mid i\in J\}$ is a flag in $\Gamma^n$ then, for $r\notin J$, the element $(x_{r1},\ldots,x_{rn})$ is in the residue of $F''$ if and only if, for each $k\leq n$, $x_{rk}$ is in the residue of the flag $\{y_{ik}\}_{i\in J}$ of $\Gamma$. It follows that the residue of $F''$ is complete bipartite if and only if the residue of $\{y_{i\ell}\}_{i\in J}$ is complete bipartite for each $\ell$, and so the basic diagram for $\Gamma^n$ is the same as that for $\Gamma$.
\end{proof}

\subsection{A second product construction}
Construction \ref{increasing1s} below shows how to build a \ourgeoms{G}{\Omega} of rank at most $\lfloor n/2\rfloor -1$ where each set of elements of a given type is a copy of $\Omega=\Delta^n$ and $G = H \Wr S_n$ with $H$ primitive on $\Delta$.  As mentioned in Section \ref{QPPA}, there exist primitive groups of types PA, HC, CD, HA and TW which can be viewed in this manner.  Construction \ref{increasing1s} is sufficiently generic that it enables us to build geometries of unbounded rank for certain examples of primitive groups of each of these O'Nan-Scott types.

\begin{constr} \label{increasing1s}
{\em
Let $n$ be a positive integer, let $H$ be a primitive subgroup of $S_m$ acting on a set $\Delta$, and let $G = H \Wr S_n$, acting on $\Omega := \Delta^n$ in its product action (see Section \ref{QPgroups}).  Let $\alpha$ and $\beta$ be distinct elements of $\Delta$, and for $1 \leq c \leq \lfloor n/2\rfloor-1$, let $$x_{c} := (\underbrace{\alpha,\ldots,\alpha}_{2c},\underbrace{\beta,\ldots,\beta}_{n-2c}).$$
Let $\Sigma = \{x_{c} \, | \, 1 \leq c \leq \lfloor n/2\rfloor-1\}$.  Let $X_1 = \Omega$, let $\Gamma_1 = (X_1,\ast,t)$, where $\ast$ consists of the pairs $(x,x)$ for $x \in X_1$, and $t(x) = 1$ for all $x \in X_1$, and let $K_1 = \{x_1\}$.  For $b$ with $1 < b \leq \lfloor n/2\rfloor-1$, suppose we have constructed $\Gamma_{b-1}$ and chamber $K_{b-1} = \{x_1, \ldots, x_{b-1}\}$.  Let $\Gamma_b = \Inc(\Gamma_{b-1},G,K_{b-1},\Omega,x_{b},b)$ as in Construction \ref{balloon}, and $K_b = K_{b-1} \cup \{x_{b}\}$.
}
\end{constr}

Lemmas \ref{difficultorbits2}, \ref{dagconditions} and Corollary \ref{difficultorbitscor} are technical results needed for the proof of Lemma \ref{increasing1slem}, which states that Construction \ref{increasing1s} yields a \ourgeoms{G}{\Omega} of rank $b$.
Given $x \in \Omega = \Delta^n$, $\gamma \in \Delta$, and $1 \leq i < j \leq n$, we write $\nong_{[i,j]}(x)$ to mean the number of entries of $x$ in coordinates $i$ to $j$ not equal to $\gamma$.

\begin{lemma} \label{difficultorbits2}
Let $\Lambda_i$ and $\Lambda_j$ be subsets of $\Omega = \Delta^n$ with $i < j \leq n$, and let $\alpha,\beta \in \Omega$.  Let $a > j$, and assume that for $\ell \in\{ i, j\}$, $\Lambda_\ell$ consists of the $n$-tuples $x$ such that $$\nona_{[1,\ell]}(x) + \nonb_{[\ell+1,n]}(x) = a-\ell.$$  If $x \in \Lambda_i \cap \Lambda_j$ then entries $i+1$ to $j$ of $x$ are all equal to $\alpha$.
\end{lemma}
\begin{proof}
We have $x \in \Lambda_i \cap \Lambda_j$ if and only if, taking $\ell$ equal to either $i$ or $j$, $\nona_{[1,\ell]}(x) + \nonb_{[\ell+1,n]}(x) = a-\ell$.  So we have 
\begin{itemize}
\item[(a)] $\nona_{[1,i]}(x) = a - i - \nonb_{[i+1,n]}(x)$, and
\item[(b)] $\nona_{[1,j]}(x) = a - j - \nonb_{[j+1,n]}(x)$.
\end{itemize}
Furthermore, since $i < j$
\begin{itemize}
\item[(c)] $\nonb_{[i+1,n]}(x) =  \nonb_{[i+1,j]}(x) + \nonb_{[j+1,n]}(x)$, and
\item[(d)] $\nona_{[1,i]}(x) = \nona_{[1,j]}(x) - \nona_{[i+1,j]}(x)$.
\end{itemize}

Substituting (a) into (d) gives 
$$a - i - \nonb_{[i+1,n]}(x) = \nona_{[1,j]}(x) - \nona_{[i+1,j]}(x),$$ 
and then using (c), we get 
$$a - i - \nonb_{[i+1,j]}(x) - \nonb_{[j+1,n]}(x) = \nona_{[1,j]}(x) - \nona_{[i+1,j]}(x).$$
Now, using (b) to replace the $\nona_{[1,j]}(x)$ term gives
$$a - i - \nonb_{[i+1,j]}(x) - \nonb_{[j+1,n]}(x) = a - j - \nonb_{[j+1,n]}(x) - \nona_{[i+1,j]}(x).$$
After cancelling terms and rearranging we are left with 
$$\nonb_{[i+1,j]}(x) - \nona_{[i+1,j]}(x) = j - i.$$  Since $\nonb_{[i+1,j]}(x) \leq j - i$ it follows that $\nonb_{[i+1,j]}(x) = j - i$ and $\nona_{[i+1,j]}(x) = 0$.  Hence entries $i+1$ to $j$ of $x$ are all $\alpha$.
\end{proof}

Let $G_0 := H_\beta \Wr S_n$ (the stabiliser in $G$ of the element $(\beta,\ldots,\beta) \in \Omega)$.  Let $h \in H$ such that $\beta^h = \alpha$.  Then for each $x_s \in \Sigma$ we have $x_s = (\beta,\ldots,\beta)^{h_s}$ where 
\begin{eqnarray}
h_s & = & (\underbrace{h,\ldots,h}_{2s},\underbrace{1_H,\ldots,1_H}_{n-2s}) \in H^n < G \label{hs}
\end{eqnarray}
and moreover, $G_{x_s} = h_s^{-1} G_0 h_s$.

\begin{lemma} \label{dagconditions}
Let $x_s$, $x_a \in \Sigma$ with $s < a$.  Then $x_a^{G_{x_s}}$ consists only of vectors $y$ such that $\nona_{[1,2s]}(y) + \nonb_{[2s+1,n]}(y) = 2a-2s$.
\end{lemma}
\begin{proof}
Let $y \in x_a^{G_{x_s}}$, and let $h_s$ be as in (\ref{hs}) where $\beta^h = \alpha$.  Then $y = x_a^{h_s^{-1} g h_s}$ for some $g = (t_1,\ldots,t_n)\sigma \in H_\beta \Wr S_n = G_0$.  We have
$$y =  (\underbrace{\beta,\ldots,\beta}_{2s},\underbrace{\alpha^{t_{2s+1}},\ldots,\alpha^{t_{2a}}}_{2a-2s},\underbrace{\beta,\ldots,\beta}_{n-2a})^{\sigma h_s}.$$
Since for all $i$, $\alpha^{t_i} \neq \beta$, and since $\sigma$ only permutes entries, the $n$-tuple
$$v := (\underbrace{\beta,\ldots,\beta}_{2s},\underbrace{\alpha^{t_{2s+1}},\ldots,\alpha^{t_{2a}}}_{2a-2s},\underbrace{\beta,\ldots,\beta}_{n-2a})^\sigma$$ has exactly $2a-2s$ entries not equal to $\beta$.  Let $d = \nonb_{[1,2s]}(v)$.  Then $d \leq 2a-2s$, and exactly $2a-2s-d$ entries from $2s+1$ to $n$ are $\nonb$; so $\nonb_{[2s+1,n]}(v) = 2a-2s-d$.  When we apply $h_s$ to $v$ (to obtain $y$), the entries of $v$ from $1$ to $2s$ equal to $\beta$ become $\alpha$, and each of the remaining $d$ entries becomes $\alpha^{t_ih_s}$, for some $i$, and as $\alpha^{t_i}\neq \beta$, we have $\alpha^{t_ih_s}\neq \alpha$. Hence $\nona_{[1,2s]}(y) = d$.  The entries $2s+1$ to $n$ of $v$ are unchanged by $h_s$, so $\nonb_{[2s+1,n]}(y) = \nonb_{[2s+1,n]}(v) = 2a-2s-d$.  It follows that $\nona_{[1,2s]}(y) + \nonb_{[2s+1,n]}(y) = 2a-2s$.
\end{proof}

\begin{corollary} \label{difficultorbitscor}
Let $x_i, x_j,x_a\in\Sigma$ with $2 \leq i < j < a \leq n/2-1$, and let $y \in x_a^{G_{x_i}} \cap x_a^{G_{x_j}}$.  Then entries $2i+1$ to $2j$ of $y$ are all $\alpha$.
\end{corollary}
\begin{proof}
By Lemma \ref{dagconditions}, the condition of Lemma \ref{difficultorbits2} holds for the pair $(2i,2j)$. Hence by Lemma \ref{difficultorbits2}, the entries $2i+1$ to $2j$ of $y$ are all equal to $\alpha$.
\end{proof}

\begin{lemma} \label{increasing1slem}
For $1 \leq b \leq \lfloor n/2\rfloor-1$, let $\Gamma_b$ be the pregeometry given by Construction {\rm\ref{increasing1s}}.  Then
\begin{itemize}
\item[(a)] $\Gamma_b$ is a geometry of rank $b$ and $G$ is flag-transitive on $\Gamma_b$,
\item[(b)] $\Gamma_b$ is thick if $b > 1$, and
\item[(c)] for $b > 1$ the rank $2$ truncations of $\Gamma_b$ are connected (and hence $\Gamma_b$ is a \ourgeoms{G}{\Omega}).
\end{itemize}
\end{lemma}
\begin{proof}
Part(a): The proof is by induction on $b$.  Since for $\Gamma_1$, $\ast$ is the set of pairs $(x,x)$ for $x \in X_1$, the statement is trivially true for $\Gamma_1$.  Suppose that $1 < b \leq \lfloor n/2\rfloor -1$ and that $\Gamma_{b-1}$ is a geometry with $G$ flag-transitive on $\Gamma_{b-1}$.  We check that the condition in Lemma \ref{balloonchambtrans}(b) holds for $\Gamma_b = \Inc(\Gamma_{b-1},G,K_{b-1},\Omega,x_{b},b)$. Recall that $K_{b-1} = \{x_1,\ldots,x_{b-1}\}$. Let $Q$ be a non-empty subset of $K_{b-1}$.  We need to show that $\bigcap_{x_s \in Q} G_{x_s}$ is transitive on $\bigcap_{x_s \in Q} x_b^{G_{x_s}}$.  

Let $i$ be the smallest subscript in $Q$ and $j$ the largest.  If $i = j$ then $G_{x_i} = \bigcap_{x_s \in Q} G_{x_s}$ is transitive on $x_b^{G_{x_i}} = \bigcap_{x_s \in Q} x_b^{G_{x_s}}$.  So we may assume that $i < j$.  It follows from Corollary \ref{difficultorbitscor} that:

\begin{eqnarray}
& & \text{all elements of $\bigcap_{x_s \in Q} x_b^{G_{x_s}}$ have entries $2i+1$ to $2j$ equal to $\alpha$.} \label{i+1toj1}
\end{eqnarray}

To show that $\bigcap_{x_s \in Q} G_{x_s}$ is transitive on $\bigcap_{x_s \in Q} x_b^{G_{x_s}}$, let $u$ and $u'$ be two $n$-tuples in $\bigcap_{x_s \in Q} x_b^{G_{x_s}}$. We find an element of  $\bigcap_{x_s \in Q} G_{x_s}$ mapping $u$ to $u'$.  As before, let $G_0$ be the stabiliser in $G$ of $(\beta,\ldots,\beta)$, and let $h \in H$ such that $\beta^h = \alpha$.  Then $G_0 = H_\beta \Wr S_m$, and for each $s$, $G_{x_s} = {h_s^{-1} G_0 h_s}$, with $h_s$ as in (\ref{hs}).

Thus for each $x_s \in Q$, each $n$-tuple in $x_b^{G_{x_s}}$ can be expressed as $x_b^{h_s^{-1} g h_s}$ for some $g \in G_0$.  Hence there exist $g,g' \in G_0$ such that $u = x_b^{h_j^{-1}gh_j}$ and $u' = x_b^{h_j^{-1}g'h_j}$, with $j$ as above.

From the definitions of $x_b$, $h_j$ and $g$ we observe that 
$$x_b^{h_j^{-1} g} = (\underbrace{\beta,\ldots,\beta}_{2j},\underbrace{\alpha,\ldots,\alpha}_{2b-2j},\underbrace{\beta,\ldots,\beta}_{n-2b})^{g}
$$
so $x_b^{h_j^{-1} g}$ has exactly $2b-2j$ non-$\beta$ entries, and these entries are all in $\alpha^{H_\beta}$; the same holds for  $x_b^{h_j^{-1}g'}$.  Furthermore, since by (\ref{i+1toj1}) the entries $2i+1$ to $2j$ of $u$ and $u'$ are all $\alpha$, and since $\beta^h = \alpha$, the entries $2i+1$ to $2j$ of $x_b^{h_j^{-1}g}$ and $x_b^{h_j^{-1}g'}$ are all equal to $\beta$.  Thus there exists an element $z$ of $G_0 = H_\beta \Wr S_n$ which maps $x_b^{h_j^{-1}g}$ to $x_b^{h_j^{-1}g'}$ and which has the form 
$$z = (z_1,\ldots,z_{2i},\underbrace{1_H,\ldots,1_H}_{2j-2i},z_{2j+1},\ldots,z_n)\sigma$$
where $\ell^\sigma = \ell$ for $2i+1 \leq \ell \leq 2j$.  We then have $x_b^{h_j^{-1}g h_j h_j^{-1} z h_j} = x_b^{h_j^{-1} g' h_j}$; that is, the element $h_j^{-1} z h_j$ of $G_{x_j}$ maps $u$ to $u'$.  Now, let $x_s \in Q$.  Then $i \leq s \leq j$ and we have 
\begin{eqnarray*}
x_s^{h_j^{-1}z h_j} & = & (\underbrace{\alpha,\ldots,\alpha}_{2s},\underbrace{\beta,\ldots,\beta}_{n-2s})^{h_j^{-1} z h_j}\\
& = & (\underbrace{\beta,\ldots,\beta}_{2s},\underbrace{\beta^{h^{-1}},\ldots,\beta^{h^{-1}}}_{2j-2s},\underbrace{\beta,\ldots,\beta}_{n-2s})^{z h_j}.
\end{eqnarray*}
Now $z$ fixes and acts trivially on each of the coordinates $2i+1$ to $2j$, and since $i \leq s$ and $z \in H_\beta \Wr S_n$ we obtain that
\begin{eqnarray*}
x_s^{h_j^{-1}z h_j} & = & (\underbrace{\beta,\ldots,\beta}_{2s},\underbrace{\beta^{h^{-1}},\ldots,\beta^{h^{-1}}}_{2j-2s},\underbrace{\beta,\ldots,\beta}_{n-2s})^{h_j}\\
& = & (\underbrace{\alpha,\ldots,\alpha}_{2s},\underbrace{\beta,\ldots,\beta}_{n-2s})=x_s.
\end{eqnarray*}
Hence $h_j^{-1}z h_j\in G_{x_s}$ for all $x_s \in Q$.  Thus we have shown that $\bigcap_{x_s \in Q} G_{x_s}$ is transitive on $\bigcap_{x_s \in Q} x_b^{G_{x_s}}$, and so by Lemma \ref{balloonchambtrans}, $\Gamma_b$ is a geometry with $G$ flag-transitive on $\Gamma_b$.  By induction, the result holds for each $b \leq \lfloor n/2\rfloor-1$.

\nl
Part (b):  Let $i$ be a type in $\Gamma_b$, so $1 \leq i \leq b$.  Let $K_b = \{x_1,\ldots,x_{b}\}$ be the chamber of $\Gamma_b$ given by Construction \ref{increasing1s}, and write $K' = K_b \backslash \{x_{i}\}$.  Then all elements of $K'$ are such that entries $2i-1$ to $2i+2$ are either all $\alpha$ or all $\beta$.  Hence $G_{(K')}$ contains the subgroup $\Sym(\{2i-1,\ldots,2i+2\}) \leq S_n$ acting on the coordinates.  The orbit of
$$x_{i} = (\alpha,\ldots,\alpha,\underbrace{\alpha,\alpha,\beta,\beta,}_{\mathrm{entries} \; 2i-1 \; \mathrm{to} \; 2i+2}\beta,\ldots,\beta)$$
under the subgroup $\Sym(\{2i-1,\ldots,2i+2\})$ contains six distinct $n$-tuples.  Hence $|x_{i}^{G_{(K')}}| \geq 6$, and so the co-rank $1$ flag $K'$ is contained in at least 6 chambers.  Since $G$ is flag-transitive on $\Gamma_b$, it follows that $\Gamma_b$ is thick.

\nl
Part (c): Let $i,j$ be distinct types. As $G$ is primitive, we have $\langle G_{x_i}, G_{x_j} \rangle = G$, and it follows from Lemma \ref{balloonlem}(b) that the $\{i,j\}$-truncation is connected.
\end{proof}

We now work towards determining the diagram for the geometries arising from Construction \ref{increasing1s}. 

\begin{lemma}
\label{lem:GF}
Let $\Gamma_b$ be the geometry yielded by Construction \ref{increasing1s} with $2\leq b\leq \lfloor n/2\rfloor -1$, and let $K_b$ be the chamber $\{x_1,\ldots,x_b\}$. Let $F\subset K_b$ be a flag of rank at least 2, and let $r_1=\min\{i\mid x_i\in F\}$ and $r_2=\max\{i\mid x_i\in F\}$. Then $(g_1,\ldots,g_n)\sigma\in G_F$ if and only if the following conditions all hold:
\begin{enumerate}
 \item[(i)] $\alpha^{g_k}=\alpha$ and $\beta^{g_k}=\beta$ for all $k$ such that $2r_1+1\leq k\leq 2r_2$;
 \item[(ii)] $\sigma$ fixes $\{2r_1+1,\ldots,2i\}$ for each $x_i\in F$;
 \item[(iii)] for all $k\leq 2r_1$, $\alpha^{g_k}\in\{\alpha,\beta\}$ and if $\alpha^{g_k}=\beta$ then $k^\sigma\geq 2r_2+1$ while if $\alpha^{g_k}=\alpha$ then $k^\sigma\leq 2r_1$;
 \item[(iv)] for all $k\geq 2r_2+1$, $\beta^{g_k}\in\{\alpha,\beta\}$ and if $\beta^{g_k}=\alpha$ then $k^{\sigma}\leq 2r_1$ while if $\beta^{g_k}=\beta$ then $k^\sigma\geq 2r_2+1$.
\end{enumerate}
Moreover, $G_F$ induces $\Sym(\{1,2,\ldots,2r_1,2r_2+1,2r_2+2,\ldots,n\})$ on $\{1,2,\ldots,2r_1,2r_2+1,2r_2+2,\ldots,n\}$.
\end{lemma}
\begin{proof}
Note that $(g_1,\ldots,g_n)\sigma$ fixes $x_i$ if and only if the following two conditions both hold:
\begin{enumerate}
 \item[(a)] for all $k\leq 2i$, $\alpha^{g_k}\in\{\alpha,\beta\}$  and if $\alpha^{g_k}=\beta$ then $k^\sigma\geq 2i+1$ while if $\alpha^{g_k}=\alpha$ then $k^\sigma\leq 2i$.
 \item[(b)] for all $k\geq 2i+1$, $\beta^{g_k}\in\{\alpha,\beta\}$  and if $\beta^{g_k}=\alpha$ then $k^\sigma\leq 2i$ while if $\beta^{g_k}=\beta$ then $k^\sigma\geq 2i+1$.
\end{enumerate}
Thus any element satisfying the four conditions of the lemma fixes $F$. Conversely, suppose $g=(g_1,\ldots,g_n)\sigma$ fixes $F$ and let $k$ be such that $2r_1+1\leq k\leq 2r_2$.  Since $g$ fixes $x_{r_1}$ we have that $\beta^{g_k}\in\{\alpha,\beta\}$ and if $\beta^{g_k}=\alpha$ then $k^\sigma\leq 2r_1$.   Now $g$ also fixes $x_{r_2}$ and so $\alpha^{g_k}\in\{\alpha,\beta\}$. If $\alpha^{g_k}=\beta$ then $k^\sigma\geq 2r_2+1$ and $\beta^{g_k}=\alpha$. However, since $g$ fixes $x_{r_1}$ the fact that $\beta^{g_k}=\alpha$ is meant to imply that $k^\sigma\leq 2r_1$, a contradiction. Thus $\alpha^{g_k}=\alpha$ and $\beta^{g_k}=\beta$ and part (i) holds. Part (ii) then follows from the fact that $g$ fixes $x_{r_1}$ (condition (b)) and $x_i$ (condition (a)). Part (iii) follows from condition (a) applied to $x_{r_1}$ and the fact that $\sigma$ fixes $\{2r_1+1,\ldots,2r_2\}$ while part (iv) follows from condition (b) applied to $x_{r_2}$ and part (ii).

By part (ii), $G_F$ fixes $X=\{1,2,\ldots,2r_1,2r_2+1,2r_2+2,\ldots,n\}$ setwise. Also by parts (iii) and (iv), $G_F$ contains all elements $(1,1,\ldots,1)\sigma$ with $\sigma\in\Sym(\{1,2,\ldots,2r_1\})\times\Sym(\{2r_2+1,2r_2+2,\ldots,n\})$. Since $G_F$ contains the element $((\alpha,\beta),1,\ldots,1,(\alpha,\beta))(1,n)$, it follows that $G_F$ acts transitively and also primitively on the set $X$. As $G_F$ contains an element inducing a 2-cycle on $X$, \cite[Theorem 3.3A]{DM} implies that $G_F^{X}=\Sym(X)$.
\end{proof}

Before determining the diagram of the geometries yielded by Construction \ref{increasing1s}, we describe some rank 2 geometries. For each such geometry, we provide a diagram 
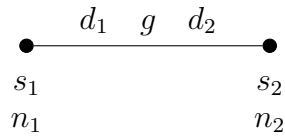
\begin{center}
\begin{tikzpicture}[scale=0.4]

 \node (A) at (0,0) [circle, draw, fill=black!100, inner sep=0pt, minimum width=5pt,label=below:$\begin{array}{c} s_1 \\ n_1 \end{array}$] {}; 
\node (B) at (8,0)[circle, draw, fill=black!100, inner sep=0pt, minimum width=5pt, label=below:$\begin{array}{c} s_2 \\ n_2 \end{array}$] {};

\path (A) edge node [above] {$d_1$ \,\, $g$ \,\, $d_2$} (B);

\end{tikzpicture}
\end{center}
where for $i=1$ and $2$, $n_i$ denotes the number of elements of type $i$, $s_i+1$ is the number of elements of type $3-i$ incident to an element of type $i$, $d_i$ is the largest distance an element can be from an element of type $i$ in the incidence graph of the geometry, and $g$ denotes the gonality of the geometry, that is, $2g$ is the length of the smallest cycle in the incidence graph.

We use $U_{a,b}(m)$ to denote the rank 2 geometry whose elements of type 1 are the $a$-subsets of $\{1,\ldots,m\}$, whose elements of type 2 are the $b$-subsets of $\{1,\ldots,m\}$ and incidence is given by inclusion. The diagram for $U_{2,4}(6)$ is 
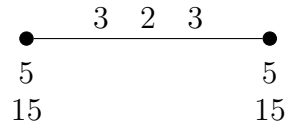
\begin{center}
\begin{tikzpicture}[scale=0.4]

 \node (A) at (0,0) [circle, draw, fill=black!100, inner sep=0pt, minimum width=5pt,label=below:$\begin{array}{c} 5\\  15 \end{array}$] {}; 
\node (B) at (8,0)[circle, draw, fill=black!100, inner sep=0pt, minimum width=5pt, label=below:$\begin{array}{c} 5 \\  15 \end{array}$] {};

\path (A) edge node [above] {$3$ \,\, $2$ \,\, $3$} (B);

\end{tikzpicture}
\end{center}

We use $U_{a,b}(m,\delta)$ to denote the rank 2 geometry whose elements of type 1 are the $a$-subsets of $\{1,\ldots,m\}$ with each element of the $a$-subset coloured from a palette of size $\delta$, and whose elements of type 2 are the $b$-subsets of $\{1,\ldots,m\}$ with each element of the $b$-subset coloured from the same palette of size $\delta$. A coloured $a$-subset is incident with a coloured $b$-subset if one is contained in the other. Note that $U_{a,b}(m,1)=U_{a,b}(m)$. The diagram for $U_{2,4}(m,\delta)$ with $\delta\geq 2$ and $m\geq 5$ is
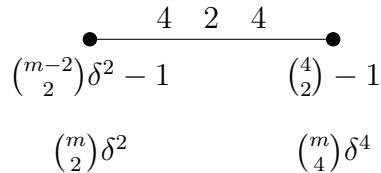
\begin{center}
\begin{tikzpicture}[scale=0.4]

 \node (A) at (0,0) [circle, draw, fill=black!100, inner sep=0pt, minimum width=5pt,label=below:$\begin{array}{c} \binom{m-2}{2}\delta^{2}-1 \\ \\\binom{m}{2}\delta^2 \end{array}$] {}; 
\node (B) at (8,0)[circle, draw, fill=black!100, inner sep=0pt, minimum width=5pt, label=below:$\begin{array}{c} \binom{4}{2}-1\\ \\ \binom{m}{4}\delta^4 \end{array}$] {};

\path (A) edge node [above] {$4$ \,\, $2$ \,\, $4$} (B);

\end{tikzpicture}
\end{center}
Note that given a 2-set $\{x,y\}$ with $x$ and $y$ both coloured blue, the elements incident with it are the coloured 4-sets containing it while the elements at distance 2 are the coloured 2-sets for which either the underlying 2-set is disjoint from $\{x,y\}$, or any $x$ or $y$ it contains is coloured blue. The elements at distance 3 are the remaining 4-sets, that is those containing at least one $x$ or $y$ coloured not blue, and the elements at distance 4 are the remaining coloured 2-sets, that is those containing at least one $x$ or $y$ coloured not blue.  Given a 4-set $\{x,y,u,v\}$ with all elements coloured blue, the elements incident with it are the 2-sets it contains, the elements at distance 2 are the 4-sets containing at least two of $\{x,y,u,v\}$ coloured blue, the elements are distance 3 are the remaining coloured 2-sets and the elements at distance 4 are the remaining 4-sets.

Finally, we use $\overline{U}_{a,b}(m,\delta)$ to denote the rank 2 geometry whose elements of type 1 are the $a$-subsets of $\{1,\ldots,m\}$ with each element of the $a$-subset coloured from a palette of size $\delta$, and whose elements of type 2 are the $b$-subsets of $\{1,\ldots,m\}$ with each element of the $b$-subset coloured from the same palette of size $\delta$.  A coloured $a$-subset is incident with a coloured $b$-subset if their underlying subsets are disjoint. The geometry $\overline{U}_{2,2}(m,\delta)$ with $m\geq 6$ and $\delta\geq 2$ has diagram
\begin{center}
\begin{tikzpicture}[scale=0.4]

 \node (A) at (0,0) [circle, draw, fill=black!100, inner sep=0pt, minimum width=5pt,label=below:$\begin{array}{c} \binom{m-2}{2}\delta^2-1 \\ \\ \binom{m}{2}\delta^2 \end{array}$] {}; 
\node (B) at (8,0)[circle, draw, fill=black!100, inner sep=0pt, minimum width=5pt, label=below:$\begin{array}{c} \binom{m-2}{2}\delta^2-1 \\ \\  \binom{m}{2}\delta^2 \end{array}$] {};

\path (A) edge node [above] {$3$ \,\, $2$ \,\, $3$} (B);

\end{tikzpicture}
\end{center}
Note that given a 2-set $\alpha=\{x,y\}$ coloured blue, the elements incident to it are the 2-sets of the other type that contain neither $x$ nor $y$, the elements at distance 2 all the 2-sets of the same type as $\alpha$ other than $\alpha$, and the elements at distance 3 are all the remaining 2-sets of the other type.

\begin{theorem}
\label{thm:proddiag}
Let $\Gamma_b$ be the geometry yielded by Construction \ref{increasing1s} with $2\leq b\leq \lfloor n/2\rfloor -1$. Then the diagram of $\Gamma_b$ is 
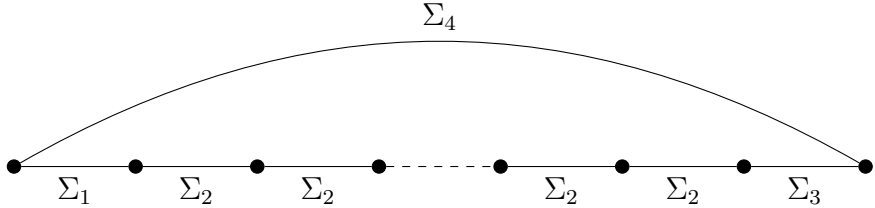
\begin{center}
\begin{tikzpicture}[scale=0.4]

 \node (A) at (0,0) [circle, draw, fill=black!100, inner sep=0pt, minimum width=5pt]{}; 
\node (B) at (4,0)[circle, draw, fill=black!100, inner sep=0pt, minimum width=5pt] {}; 
\node (C) at (8,0)[circle, draw, fill=black!100, inner sep=0pt, minimum width=5pt] {} ;
\node (D) at (12,0)[circle, draw, fill=black!100, inner sep=0pt, minimum width=5pt]{} ; 
\node (E) at (16,0)[circle, draw, fill=black!100, inner sep=0pt, minimum width=5pt]{} ; 
\node (F) at (20,0) [circle, draw, fill=black!100, inner sep=0pt, minimum width=5pt] {}; 
\node (G) at (24,0) [circle, draw, fill=black!100, inner sep=0pt, minimum width=5pt] {};
\node (H) at (28,0) [circle, draw, fill=black!100, inner sep=0pt, minimum width=5pt] {};

\path (A) edge node [below] {$\Sigma_1$} (B);

\path (B) edge node [below] {$\Sigma_2$} (C);
\path (C) edge node [below] {$\Sigma_2$} (D);

\draw [dashed] (D) to (E);

\path (E) edge node [below] {$\Sigma_2$} (F);

\path (F) edge node [below] {$\Sigma_2$} (G);

\path (G) edge node [below] {$\Sigma_3$} (H);

\path (A) edge [bend left]  node [auto] {$\Sigma_4$} (H);

\end{tikzpicture}
\end{center}
where $\Sigma_1=U_{4,2}(n-2b+6,|\Delta|-1)$,  $\Sigma_2=U_{2,4}(6)$, $\Sigma_3=U_{2,4}(n-2b+6,|\Delta|-1)$ and
$\Sigma_4= \overline{U}_{2,2}(n-2b+6,|\Delta|-1)$.
\end{theorem}
\begin{proof}
Let $F$ be a flag of rank $b-2$.  By Lemma \ref{increasing1slem}, $G$ is flag-transitive on $\Gamma_b$ so we may 
assume that $F\subset K_b=\{x_1,\ldots,x_b\}$. Let $x_{u}$ and $x_{v}$ be the two elements of $K_b\backslash F$ with $u<v$. 
Since $G$ is flag-transitive on $\Gamma_b$, the residue of $F$ is $(x_{u})^{G_F}\cup (x_{v})^{G_F}$ and the set of 
elements in this residue incident with $x_{u}$ is $(x_{v})^{G_{F\cup\{x_{u}\}}}$ while the set of elements incident 
with $x_{v}$ is $(x_{u})^{G_{F\cup\{x_{v}\}}}$. We split our analysis into several cases:

\medskip
\underline{$1<u<v<b$:} 
By Lemma \ref{lem:GF}, if $(g_1,\ldots,g_n)\sigma\in G_F$ then
\begin{enumerate}
 \item[(i)] $\alpha^{g_k}=\alpha$ and $\beta^{g_k}=\beta$ for all $k$ such that $3\leq k\leq 2b$.
 \item[(ii)] for $k\in\{1,2\}$, $\alpha^{g_k}\in\{\alpha,\beta\}$  and if $\alpha^{g_k}=\beta$ then $k^\sigma\geq 2b+1$ while if $\alpha^{g_k}=\alpha$ then $k^\sigma\in\{1,2\}$.
 \item[(iii)] for all $k\geq 2b+1$, $\beta^{g_k}\in\{\alpha,\beta\}$  and if $\beta^{g_k}=\alpha$ then $k^{\sigma}\in\{ 1,2\}$ while if $\beta^{g_k}=\beta$ then $k^\sigma\geq 2b+1$.
\end{enumerate}

Suppose first that $v\neq u+1$. Then $G_F$ induces 
$$\Sym(\{1,2,2b+1,2b+2,\ldots,n\}\times\Sym(\{2u-1,2u,2u+1,2u+2\})\times $$
$$\Sym(\{2v-1,2v,2v+1,2v+2\})\times S_2^{b-5}$$ on $\{1,\ldots,n\}$. It follows that
$(x_{u})^{G_F}$ consists of all $n$-tuples with 
entries $\{1,\ldots,2(u-1)\}$ equal to $\alpha$, entries $\{2u+3,\ldots,n\}$ equal to $\beta$ and precisely two 
of the entries $\{2u-1,2u,2u+1,2u+2\}$ equal to $\alpha$ and the remaining two equal to $\beta$.  Hence $|(x_{u})^{G_F}|=6$. Similarly $|(x_{v})^{G_F}|=6$. Moreover, since $x_{u-1},x_{u+1}\in F$ and $x_u\notin F$, Lemma \ref{lem:GF}, implies that $G_{F\cup\{x_{v}\}}$ 
contains $\Sym(\{2u-1,2u,2u+1,2u+2\})$ and so $x_{v}$ is incident with each element of $(x_{u})^{G_F}$. Thus the residue of $F$ has incidence graph the complete bipartite graph $K_{6,6}$.

When $v=u+1$, $G_F$ induces $$\Sym(\{1,2,2b+1,2b+2,\ldots,n\}\times\Sym(\{2u-1,2u,2u+1,2u+2,2u+3,2u+4\})\times S_2^{b-4}$$ on $\{1,\ldots,n\}$.
Then $(x_{u})^{G_F}$ consists of all $n$-tuples with 
entries $\{1,\ldots,2(u-1)\}$ equal to $\alpha$, entries $\{2u+5,\ldots,n\}$ equal to $\beta$ and precisely two 
of the entries $\{2u-1,2u,2u+1,2u+2,2u+3,2u+4\}$ equal to $\alpha$ and the remaining four equal to $\beta$. 
Thus $|(x_{u})^{G_F}|=15$. Similarly $|(x_{v})^{G_F}|$ consist of all $n$-tuples with 
entries $\{1,\ldots,2(u-1)\}$ equal to $\alpha$, entries $\{2u+5,\ldots,n\}$ equal to $\beta$ and precisely four 
of the entries $\{2u-1,2u,2u+1,2u+2,2u+3,2u+4\}$ equal to $\alpha$ and the remaining two equal to $\beta$. Now if $(g_1,\ldots,g_n)\sigma\in G_{F\cup\{x_{v}\}}$, then by Lemma \ref{lem:GF}, $\sigma$ fixes $\{2u-1,2u,2u+1,2u+2\}$ and $G_{F\cup\{x_{v}\}}$ includes $\Sym(\{2u-1,2u,2u+1,2u+2\})$. Thus in the residue of $F$, $x_{v}$ is incident with those elements of $(x_{u})^{G_F}$ for which precisely two of the entries from $\{2u-1,2u,2u+1,2u+2\}$ are equal to $\alpha$. Similarly, the elements of the residue of $G$ incident with $x_{u}$ are those elements of $(x_{v})^{G_F}$ such that the entries $\{2u-1,2u\}$ are equal to $\alpha$. Thus the residue of $F$ is isomorphic to $U_{2,4}(6)$.

\medskip
\underline{$u=1$ and $2<v<b$:} By Lemma \ref{lem:GF}, $G_F$ induces $$\Sym(\{1,2,3,4,2b+1,2b+2,\ldots,n\})\times S_2^{b-4}\times \Sym(\{2v-1,2v,2v+1,2v+2\})$$ on $\{1,\ldots,n\}$, while
$G_{F\cup\{x_1\}}$ induces $$\Sym(\{1,2,2b+1,2b+2,\ldots,n\})\times \Sym(\{3,4\})\times S_2^{b-4}\times \Sym(\{2v-1,2v,2v+1,2v+2\}).$$ Moreover, if $g=(g_1,\ldots,g_n)\sigma$ is an element of $G_F$ such that $\sigma$ fixes $\{3,4\}$ then $\alpha^{g_k}=\alpha$ for $k=3$ and $4$. Furthermore, $g$ also fixes $x_1$ if and only if $\beta^{g_k}=\beta$ for $k=3$ and $4$. Hence 
$|x_1^{G_F}|=(|\Delta|-1)^2\binom{n-2b+4}{2}$. Note that an element of $G_F$ fixes $x_{v}$ if and only if it fixes 
$\{2v-1,2v\}$ setwise. Since both $G_F$ and $G_{F\cup\{x_{u}\}}$ fix $\{2v-1,2v,2v+1,2v+2\}$ setwise and induce $\Sym(\{2v-1,2v,2v+1,2v+2\})$, it follows that 
$|(x_{v})^{G_F}|=\binom{4}{2}=6$ and $(x_{v})^{G_{F\cup\{x_{u}\}}}=(x_v)^{G_{F}}$. Thus the residue of $F$ is the 
complete bipartite graph with bipartite halves of size $(|\Delta|-1)^2\binom{n-2b+4}{2}$ and $6$.


\medskip
\underline{$1<u<b-1$ and $v=b$:} Arguing as in the previous case yields that the residue of $F$ is the complete bipartite graph with bipartite halves of size $6$ and $(|\Delta|-1)^2\binom{n-2b+4}{2}$.

%

\medskip
\underline{$u=1$ and $v=2<b$:}
By Lemma \ref{lem:GF}, 
$$\begin{array}{cl}
(G_F)^{\{1,\ldots,n\}} &=\Sym(\{1,2,3,4,5,6,2b+1,2b+2,\ldots,n\})\times S_2^{b-3} \\ (G_{F\cup\{x_1\}})^{\{1,\ldots,n\}} &=\Sym(\{1,2,2b+1,2b+2,\ldots,n\})\times \Sym(\{3,4,5,6\})\times S_2^{b-3} \\
(G_{F\cup\{x_2\}})^{\{1,\ldots,n\}} &=\Sym(\{1,2,3,4,2b+1,2b+2,\ldots,n\})\times \Sym(\{5,6\})\times S_2^{b-3}
\end{array}$$
Moreover, if $g=(g_1,\ldots,g_n)\sigma$ is an element of $G_F$ such that $\sigma$ fixes $\{3,4,5,6\}$ setwise, then $\alpha^{g_k}=\alpha$ for each $k\in\{3,4,5,6\}$. Furthermore, $g$ also fixes $x_1$ if and only if $\beta^{g_k}=\beta$ for each such $k$. Hence 
$|x_1^{G_F}|=(|\Delta|-1)^4\binom{n-2b+6}{4}$. Note that the elements of $x_1^{G_F}$ correspond to the $n$-tuples with entries $\{7,8,\ldots,2b\}$ equal to $\beta$, all but $r$ of the entries $\{2b+1,\ldots,n\}$ equal to $\beta$ for some $r$ with $0\leq r\leq 4$, and $4-r$ of the entries $\{1,2,3,4,5,6\}$ not equal to $\alpha$. In particular, the elements of $x_1^{G_F}$ correspond to 4-subsets of $\{1,2,3,4,5,6,2b+1,2b+2,\ldots,n\}$ with elements of $\{1,\ldots,6\}$ coloured by $\Delta\backslash\{\alpha\}$ and elements of $\{2b+1,2b+2,\ldots,n\}$ coloured by $\Delta\backslash\{\beta\}$.
 Similarly, $|x_2^{G_F}|=(|\Delta|-1)^2\binom{n-2b+6}{2}$ and the elements of $x_2^{G_F}$ correspond to the 2-subsets of $\{1,2,3,4,5,6,2b+1,2b+2,\ldots,n\}$ with elements of $\{1,\ldots,6\}$ coloured by $\Delta\backslash\{\alpha\}$ and elements of $\{2b+1,2b+2,\ldots,n\}$ coloured with $\Delta\backslash\{\beta\}$.

%

Now $(g_1,\ldots,g_n)\sigma\in G_{F\cup\{x_1\}}$ fixes $x_2$ if and only if $\sigma$ fixes $\{5,6\}$ setwise. Thus $|x_2^{G_{F\cup\{x_1\}}}|=\binom{4}{2}=6$ and the elements in the residue of $F$ incident with $x_1$ are those elements with the first two entries equal to $\alpha$, the last $n-6$ entries equal to $\beta$, two of the entries $\{3,4,5,6\}$ equal to $\alpha$ and the remaining two entries equal to $\beta$.  The 4-subset corresponding to $x_1$ is $\{3,4,5,6\}$ with all elements coloured by $\beta$ and the 2-subsets corresponding to the elements of $x_2^{G_{F\cup\{x_1\}}}$ are the 2-subsets of this 4-subset with the inherited colouring.

By Lemma \ref{lem:GF}, if $(g_1,\ldots,g_n)\sigma\in G_{F\cup\{x_2\}}$ such that $\sigma$ fixes $\{3,4\}$ then $\alpha^{g_k}=\alpha$ for $k=3,4$. Such an element then fixes $x_1$ if and only if it is also the case that $\beta^{g_k}=\beta$ for $k=3,4$. Hence $|(x_1)^{G_{F\cup\{x_2\}}}|=(|\Delta|-1)^2\binom{n-2b+4}{2}$.  Moreover, $(x_1)^{G_{F\cup\{x_2\}}}$ corresponds to the coloured 4-subsets containing the coloured 2-subset corresponding to $x_2$. It follows that the residue of $F$ is isomorphic to $U_{4,2}(n-2b+6,|\Delta|-1)$.

\medskip
\underline{$1<u=b-1$ and $v=b$:} Arguing as in the previous case yields that the residue of $F$ is isomorphic to $U_{2,4}(n-2b+6,|\Delta|-1)$.



\medskip
\underline{$u=1$ and $v=b$:}
By Lemma \ref{lem:GF}, 
$$\begin{array}{cl}
   (G_F)^{\{1,2,\ldots,n\}}&=\Sym(\{1,2,3,4,2b-1,2b,\ldots,n\})\times S_2^{b-3}\\
   (G_{F\cup\{x_1\}})^{\{1,2,\ldots,n\}}&=\Sym(\{1,2,2b-1,2b,\ldots,n\})\times\Sym(\{3,4\})\times S_2^{b-3}\\
   (G_{F\cup\{x_b\}})^{\{1,2,\ldots,n\}}&=\Sym(\{1,2,3,4,\ldots,n\})\times\Sym(\{2b-1,2b\})\times S_2^{b-3}
  \end{array}$$
Moreover, if $(g_1,\ldots,g_n)\sigma\in G_F$ such that $\sigma$ fixes $\{3,4\}$ then $\alpha^{g_k}=\alpha$ for $k=3,4$. Furthermore, such an element also fixes $x_1$ if and only if $\beta^{g_k}=\beta$ for $k=3,4$. Hence $|x_1^{G_F}|=(|\Delta|-1)^2\binom{n-2b+4}{2}$ and the elements of $x_1^{G_F}$ correspond to the set of 2-subsets of $\{1,2,3,4,2b-1,2b,\ldots,n\}$ with elements of $\{1,2,3,4\}$ coloured by $\Delta\backslash\{\alpha\}$ and elements of $\{2b-1,2b,\ldots,n\}$ coloured by $\Delta\backslash\{\beta\}$. Similarly, $|x_b^{G_F}|=(|\Delta|-1)^2\binom{n-2b+4}{2}$ and $x_b^{G_F}$ also corresponds to the set of 2-subsets of $\{1,2,3,4,2b-1,2b,\ldots,n\}$ with elements of $\{1,2,3,4\}$ coloured by $\Delta\backslash\{\alpha\}$ and elements of $\{2b-1,2b,\ldots,n\}$ coloured by $\Delta\backslash\{\beta\}$. Moreover, $|x_1^{G_{F\cup\{x_b\}}}|=|x_b^{G_{F\cup\{x_1\}}}|=(|\Delta|-1)^2\binom{n-2b+2}{2}$. In particular a coloured 2-subset in $x_1^{G_F}$ is incident with a coloured 2-subset in $x_b^{G_F}$ if and only if their 2-subsets are disjoint. Thus the residue of $F$ is isomorphic to $\overline{U}_{2,2}(2-2b+6,|\Delta|-1)$.


\end{proof}

\section{HS type}\label{HStype}

In this section we give a construction of a \ourgeoms{G}{\Omega} of arbitrary rank where $G=T\times T$ acts on $\Omega=T$ in a primitive action of type HS.

\begin{constr} \label{HSconstruction}
{\em
Let $m\geq 5$ be an integer, $T=A_m$, and let $G = T\times T$ acting on $\Omega=T$ by $t^{(t_1,t_2)}=t_1^{-1}tt_2$.  Let $x_0=1_T$, and for $i\in\{1,\ldots,\lfloor m/4\rfloor\}$, let $x_i = (1,2)(3,4)...(4i-1,4i) \in \Omega$, the product of $2i$ transpositions, and let $\Sigma = \{x_i \mid 0 \leq i \leq \lfloor m/4\rfloor\}$.  Let $X_0 = \Omega$, let $\Gamma_0 = (X_0,\ast,t)$, where $\ast$ consists of the pairs $(x,x)$ for $x \in X_0$, and $t(x) = 0$ for all $x \in X_0$.  Let $K_0 = \{x_0\}$.  For $b$ with $1 \leq b \leq\lfloor m/4\rfloor$, suppose we have constructed $\Gamma_{b-1}$ and a chamber $K_{b-1} = \{x_0, \ldots, x_{b-1}\}$.  Let $\Gamma_b = \Inc(\Gamma_{b-1},G,K_{b-1},\Omega,x_b,b)$ as in Construction \ref{balloon}, and $K_b = K_{b-1} \cup \{x_b\}$.
}
\end{constr}

Note that 
\begin{align}\label{eq:stab}
 G_{x_i}&=\{(t_1,t_2)\mid   t_1^{-1}x_it_2=x_i,\ \mbox{with}\ t_1, t_2\in T\}\nonumber\\ 
        &=\{(t_1,t_2)\mid  t_2=x_it_1x_i,\ \mbox{with}\ t_1, t_2\in T\}\nonumber\\
        &=\{(t,x_itx_i)\mid t\in T\}
  \end{align}
and
\begin{equation}\label{eq:orbit}
x_j^{G_{x_i}}=\{t^{-1}x_jx_itx_i\mid t\in T\}=(x_jx_i)^Tx_i. 
\end{equation}
Since $x_jx_i$ is a product of $2|j-i|$ transpositions, and all such elements are conjugate in $T$, it follows that $(x_jx_i)^T=(x_{|j-i|})^T$.

The following lemma plays a crucial role in the analysis.
\begin{lemma}
\label{lem:containstrans}
Let $y\in x_b^{G_{x_i}}\cap  x_b^{G_{x_j}}$ with $i<j<b\leq\lfloor m/4\rfloor$. Then when written as a product of disjoint cycles, $y$ includes $(4i+1,4i+2)\ldots(4j-1,4j)=x_jx_i$.
\end{lemma}
\begin{proof}
By (\ref{eq:orbit}) and the remark following it, 
there exist $y_{b-i}\in (x_{b-i})^T$ and $y_{b-j}\in (x_{b-j})^T$ such that $y=y_{b-i}x_i=y_{b-j}x_j$. Thus $y_{b-j}y_{b-i}=x_j x_i=(4i+1,4i+2)\ldots(4j-1,4j)$, and we note that the right hand side is a product of exactly $2(j-i)$ transpositions moving $4(j-i)$ points. Since $y_{b-j}$ is a product of $2(b-j)$ transpositions and $y_{b-i}$ is a product of $2(b-i)$ transpositions and $2(b-i)-2(b-j)=2(j-i)$, it follows that
$y_{b-j}y_{b-i}$ moves at least $4(j-i)$ points. Thus the $2(b-j)$ transpositions of $y_{b-j}$ are also transpositions for $y_{b-i}$, and the disjoint cycle representation for $y_{b-i}$ includes $(4i+1,4i+2)\ldots(4j-1,4j)$. Since $y=y_{b-i}x_i$ it follows that $y$ also contains $(4i+1,4i+2)\ldots(4j-1,4j)$.
\end{proof}

\begin{lemma} \label{HSconstructionlem}
For $0 \leq b \leq \lfloor m/4\rfloor$, let $\Gamma_b$ be the pregeometry given by Construction {\em \ref{HSconstruction}}. Then
\begin{itemize}
\item[(a)] $\Gamma_b$ is a geometry of rank $b+1$ and $G$ is flag-transitive on $\Gamma_b$,
\item[(b)] $\Gamma_b$ is thick if $b > 0$, and
\item[(c)] if $b > 0$ then the rank $2$ truncations of $\Gamma_b$ are connected (and hence $\Gamma_b$ is a \ourgeoms{G}{\Omega}).
\end{itemize}
\end{lemma}
\begin{proof}
Part (a): The proof is by induction on $b$.  As chambers are singletons for $\Gamma_0$, the statement is trivially true for $\Gamma_0$.  Suppose inductively that $0 < b \leq \lfloor m/4\rfloor$ and that $\Gamma_{b-1}$ is a geometry with  $G$ flag-transitive on $\Gamma_{b-1}$.  We check that the condition in Lemma \ref{balloonchambtrans}(b) holds for $\Gamma_b = \Inc(\Gamma_{b-1},G,K_{b-1},\Omega,x_{b},b)$. 

Consider $Q\subseteq \{x_0,x_1,\ldots,x_{b-1}\}$ and let $k=\min\{i:x_i\in Q\}$ and $j=\max\{i:x_i\in Q\}$. Since the condition of Lemma \ref{balloonchambtrans}(b)  holds for $|Q|=1$, by the definition of $\Gamma_b$, we assume that $|Q|\geq 2$ and so $k\neq j$.

Let $y_1,y_2\in \cap_{x_i\in Q} x_b^{G_{x_i}} \subseteq x_b^{G_{x_k}}\cap x_b^{G_{x_j}}$. Then by (\ref{eq:orbit}), there exist $z_1,z_2\in (x_{b-k})^T$ such that $y_1=z_1x_k$ and $y_2=z_2x_k$. Moreover, by Lemma \ref{lem:containstrans}, each $y_i$ contains $x':= (4k+1,4k+2)\ldots(4j-1,4j)=x_kx_j$ in its disjoint cycle representation, and since $x_k$ moves only the points $1,\dots,4k$, it follows that each $z_i$ contains $x'$ in its disjoint cycle representation. Thus, since $z_i$ is a product of $2(b-k)$ transpositions, there exists $y_i'\in \Sym(\{1,\ldots,4k,4j+1,\ldots,m\})$, a product of $2(b-k)-2(j-k)=2(b-j)$ transpositions, such that $z_i=y_i'x'$ and hence such that $y_i=y_i'x'x_k$. 

For distinct $i,\ell\in [k,j]$, we have from (\ref{eq:stab}) that 
\begin{align*}
 G_{x_i}\cap G_{x_\ell}&=\{(t,x_itx_i)\mid t\in T, x_\ell tx_\ell=x_itx_i\}\\
                    &=\{(t,x_itx_i)\mid t\in C_T(x_\ell x_i)\}.
\end{align*}
Hence 
$$
\bigcap_{x_i\in Q} G_{x_i}=\{(t,x_ktx_k)\mid t\in C_T(x_kx_i) \text{ for all } x_i\in Q\}.
$$
For each $i\in Q$, the product $x_kx_i$ fixes the set $J:=\{1,\dots,4k, 4j+1,\dots,m\}$ pointwise, and is contained in the disjoint cycle representation for $x'=x_kx_j$. It follows that the group induced on $J$ by $\cap_{x_i\in Q}C_T(x_kx_i)$ is $\Sym(J)$.
Hence, there exists $t\in \cap_{x_i\in Q}C_T(x_kx_i)$ such that
$(y_1'x')^t=y_2'x'$.  
Then $(t,x_ktx_k)\in\cap_{x_i\in Q} G_{x_i}$ and maps $y_1$ to 
\[
t^{-1}y_1x_ktx_k=t^{-1}(y_1'x'x_k)x_ktx_k = t^{-1} y_1'x'tx_k\\
                =y_2'x'x_k=y_2.
\]
Thus $\cap_{x_i\in Q}C_T(x_kx_i)$ acts transitively on $\cap _{x_i\in Q} x_b^{G_{x_i}}$. Hence the condition of  Lemma \ref{balloonchambtrans}(b) holds and so by induction, $\Gamma_b$ is a geometry of rank $b+1$ and $G$ is flag-transitive on $\Gamma_{b}$.

Part(b): Consider the chamber $K_b=\{x_0,x_1,\ldots,x_b\}$, with $b \geq 1$, and let $K'$ be any co-rank $1$ flag contained in $K_b$, with $x_\ell$ the unique point in $K_b \backslash K'$. Let $k=\min\{i:x_i\in K'\}$. Then by (\ref{eq:stab}),
$$
G_{(K')}=\{(t,x_ktx_k)\mid t\in C_T(x_kx_i) \text{ for all } x_i\in K'\}.
$$ 
Now $\cap_{x_i\in K'} C_T(x_kx_i)$ induces the wreath product $S_2\Wr S_4$ on $\{4\ell-3,4\ell-2,4\ell-1,4\ell,4\ell+1,4\ell+2,4\ell+3,4\ell+4\}$ and so $x_{\ell}x_k$ has at least 6 images under $\cap_{x_i\in K'} C_T(x_kx_i)$. Since 
$$x_\ell^{(\bigcap_{x_i\in K'}G_{x_i})}=\{(x_\ell x_k)^tx_k\mid t\in \cap_{x_i\in K'} C_T(x_kx_i)\}$$
it follows that $x_\ell$ has at least 6 images under $\cap_{x_i\in K'}G_{x_i}$ and so $K'$ is contained in at least 6 chambers. Hence $\Gamma_b$ is thick.

Part (c): Since $G$ is primitive, $\langle G_{x_i}, G_{x_j} \rangle = G$ for all distinct $x_i, x_j \in \Sigma$.  Hence by Lemma \ref{balloonlem}(b), the rank $2$ truncations are connected.
\end{proof}

\begin{remark}\label{rem:HStoSD}
Let $G$ be as in Construction \ref{HSconstruction} and let $H=\langle G,\sigma\rangle$ such that $\sigma:T\rightarrow T$, $t\mapsto t^{-1}$. Then $H$ acts primitively on $T$ with O'Nan-Scott type SD. Since each $x_i$ is an involution, $H_{x_i}=\langle G_{x_i},\sigma\rangle$. Moreover, if $y\in x_j^{G_{x_i}}=(x_jx_i)^Tx_i$ then $y=zx_i$ for some involution $z\in (x_jx_i)^T$. Now $y^{-1}=x_iz=x_izx_ix_i$ and $x_izx_i\in (x_jx_i)^T$. Hence $y^{-1}\in x_j^{G_{x_i}}$ and so $(x_j^{G_{x_i}})^\sigma=x_j^{G_{x_i}}$. It follows that $H\leqslant \Aut(\Gamma_b)$ and so $\Gamma_b$ is also an \ourgeoms{H}{\Omega} for which $H$ acts primitively of type SD on $\Omega$.
\end{remark}

\begin{remark}
\label{rem:HSdiag}
We do not determine the diagrams for the geometries yielded by  Construction \ref{HSconstruction}. Note that by \cite[Proposition 18.1]{priminc}, a primitive group with O'Nan-Scott type HS cannot contain a primitive group of any other O'Nan-Scott type, and the only way it can be contained in a primitive group of another O'Nan-Scott type (other than the full alternating or symmetric group) is as in Remark \ref{rem:HStoSD}. Thus the geometries obtained are different to those obtained from Construction \ref{increasing1s}.
\end{remark}

\section{SD type} \label{SDtype}

It follows from Remark \ref{rem:HStoSD} that for each $k$ we can construct a \ourgeoms{G}{\Omega} of rank $k$ such that $G$ acts primitively of type SD on $\Omega$ and the socle of $G$ is the direct product of 2 simple direct factors. In this section we give a construction for a group of type SD whose socle is the product of $n$ simple direct factors for any $n\geq 5$.

\begin{definition} \label{SDgroup} {\em
Let $n\geq 5$ be an integer, let $T = A_m$ for some $m \geq 5$ and let $D:=\diag(T^n) = \{(t,\ldots,t) \, | \, t \in T\} < T^n$, the `straight diagonal subgroup' of $T^n$.  Let $G = T \Wr S_n=T^n\rtimes S_n$ in its primitive simple diagonal action on the set $\Omega$ of right cosets of $D$ in $T^n$.  That is to say, elements of the `base group' $T^n$ of $G$ act by right multiplication on $\Omega$, and an element $\sigma$ in the `top group' $S_n$ acts by $\sigma : D(t_1,\ldots,t_n) \longmapsto D(t_{1\sigma^{-1}},\ldots,t_{n\sigma^{-1}})$.
}
\end{definition}

For an element $x = (t_1,\ldots,t_n) \in T^n$, the \emph{support} of $x$ is the number of entries of $x$ not equal to $1_T$.  For $\delta \in \Omega$, $\bar{\delta}$ denotes a representative of the coset $\delta$, so $\delta = D\bar{\delta}$.  We sometimes denote the coset $D\bar{\delta}$ by $[\bar{\delta}]$.

\begin{lemma} \label{uniquerep}
Each element of $\Omega$ contains at most one coset representative with support strictly less than $n/2$.
\end{lemma}
\begin{proof}
Let $x\in\Omega$ and suppose that $x=D\overline{\delta}$ with $\overline{\delta}$ of support at most $(n-1)/2$.
Any representative of the coset $x$ is of the form $v=(t,\ldots,t)\bar{\delta}$ for some $(t,\ldots,t) \in D$.  Since the support of $\bar{\delta}$ is at most $(n-1)/2$, if $t\neq 1_T$ then $v$ has strictly more than half its entries equal to $t.1_T = t\neq 1_T$ and hence has support greater than $n/2$.
\end{proof}

\begin{constr} \label{SDincreasing1s} {\em
Let $\Omega$, $T$ and $G = T \Wr S_n$ be as in Definition \ref{SDgroup}. Let $\alpha$ be an involution in $T$, and for $1 \leq c \leq \lfloor (n-1)/4\rfloor$, let $\bar{\alpha}_{c} = (\underbrace{\alpha,\ldots,\alpha}_{2c},\underbrace{1_T,\ldots,1_T}_{n-2c})$, and let
$$x_{c} = [\bar{\alpha}_{c}] = D\bar{\alpha}_{c} \in \Omega.$$
Let $\Sigma = \{x_{c} \, | \, 1 \leq c \leq \lfloor(n-1)/4\rfloor\}$, and note that by Lemma \ref{uniquerep}, the elements of $\Sigma$ are pairwise distinct (that is, $|\Sigma| = \lfloor(n-1)/4\rfloor$).  

Let $X_1 = \Omega$, let $\Gamma_1 = (X_1,\ast,t)$, where $\ast$ consists of the pairs $(x,x)$ for $x \in X_1$, and $t(x) = 1$ for all $x \in X_1$. Let $K_1 = \{x_1\}$.  For $1 < b \leq \lfloor(n-1)/4\rfloor$, suppose we have constructed $\Gamma_{b-1}$ and chamber $K_{b-1} = \{x_1, \ldots, x_{b-1}\}$.  Let $\Gamma_b = \Inc(\Gamma_{b-1},G,K_{b-1},\Omega,x_{b},b)$ as in Construction \ref{balloon}, and $K_b = K_{b-1} \cup \{x_{b}\}$.
}
\end{constr}

Let $G_0 = \langle D, S_n \rangle$ and note that for each $x_s \in \Sigma$, the point stabiliser $G_{x_s}$ is equal to $h_s^{-1} G_0 h_s$ where
\begin{eqnarray} 
h_s & = & (\underbrace{\alpha,\ldots,\alpha}_{2s},\underbrace{1_T,\ldots,1_T}_{n-2s}) \in T^n < G. \label{hs2}
\end{eqnarray} 
(This is because $G_0$ is the stabiliser in $G$ of $x_0 := [(1_T,\ldots,1_T)]$ and $x_s = x_0^{h_s}$.)

We need to understand the orbits $x_a^{G_{x_s}}$ where $x_s$, $x_a \in \Sigma$ with $s < a \leq \lfloor(n-1)/4\rfloor$.  Observe that each $y \in x_a^{G_{x_s}}$ is equal to $x_a^{h_s^{-1} \bar{t} \sigma h_s}$ for some $\bar{t}=(t,\dots,t) \in D < G$ and $\sigma \in S_n < G$.  Recall that $x_a = [\bar{\alpha}_a]$, where $\bar{\alpha}_a = (\underbrace{\alpha,\ldots,\alpha}_{2a},\underbrace{1_T,\ldots,1_T}_{n-2a})$.  We define a special representative for the coset $x_a^{h_s^{-1} \bar{t} \sigma h_s}$, namely
$$\Rep(x_a^{h_s^{-1} \bar{t} \sigma h_s}) = (\underbrace{1_T,\ldots,1_T}_{2s},\underbrace{\alpha^t,\ldots,\alpha^t}_{2a-2s},\underbrace{1_T,\ldots,1_T}_{n-2a})^\sigma(\underbrace{\alpha,\ldots,\alpha}_{2s},1_T,\ldots,1_T),$$
(where $\alpha^t = t^{-1} \alpha t$).  It can be checked easily that $\Rep(x_a^{h_s^{-1} \bar{t} \sigma h_s})$ is an element of the coset $x_a^{h_s^{-1} \bar{t} \sigma h_s}$.

\begin{lemma} \label{repdag}
Let $x_s, x_a \in \Sigma$ with $s < a \leq \lfloor(n-1)/4\rfloor$, and let $h_s^{-1} \bar{t} \sigma h_s \in G_{x_s}$.  Then the support of $\Rep(x_a^{h_s^{-1} \bar{t} \sigma h_s})$ is at most $2a \leq (n-1)/2$.  Moreover $\Rep(x_a^{h_s^{-1} \bar{t} \sigma h_s})$ is the unique representative of $x_a^{h_s^{-1} \bar{t} \sigma h_s}$ with support at most $(n-1)/2$.
\end{lemma}
\begin{proof}
First observe that $(\underbrace{1_T,\ldots,1_T}_{2s},\underbrace{\alpha^t,\ldots,\alpha^t}_{2a-2s},\underbrace{1_T,\ldots,1_T}_{n-2a})$ has support $2a-2s$.  Now, $\sigma$ permutes the entries of this $n$-tuple, preserving its support, and then multiplying on the right by $h_s$  right-multiplies the first $2s$ entries by $\alpha$.  Thus $\Rep(x_a^{h_s^{-1} \bar{t} \sigma h_s})$ has support at most $(2a-2s)+2s = 2a \leq (n-1)/2$.  The second assertion follows from Lemma \ref{uniquerep}.
\end{proof}

For $x_s, x_a \in \Sigma$ with $s < a \leq \lfloor(n-1)/4\rfloor$, and a subset $Y$ of $x_a^{G_{x_s}}$, define
$$\RS(Y) = \{\Rep(y) \, | \, y \in Y\}.$$
Note that Lemma \ref{repdag} implies that $\RS(Y)$ is well-defined and $|\RS(Y)| = |Y|$.

\begin{lemma} \label{SDdagconditions}
Let $x_s$, $x_a \in \Sigma$ with $s < a$.  Then $\RS(x_a^{G_{x_s}})$ consists only of vectors $\Rep(y)$ such that $\nona_{[1,2s]}(\Rep(y)) + \non 1_{T[2s+1,n]}(\Rep(y)) = 2a-2s$.
\end{lemma}
\begin{proof}
Let $y \in x_a^{G_{x_s}}$.  Then $y = x_a^{h_s^{-1} \bar{t} \sigma h_s}$ for some $\bar{t}=(t,\dots,t) \in \diag(T^n) < G$ and $\sigma \in S_n < G$, and $\Rep(y)=v^{\sigma}h_s$ where
$$v = (\underbrace{1_T,\ldots,1_T}_{2s},\underbrace{\alpha^t,\ldots,\alpha^t}_{2a-2s},\underbrace{1_T,\ldots,1_T}_{n-2a}).$$
Since $\sigma$ only permutes entries, the $n$-tuple $v^{\sigma}$ has exactly $2a-2s$ entries equal to $\alpha^t$, and the remaining entries all equal to $1_T$.  Let $d$ be the number of entries equal to $\alpha^t$ in coordinates $1$ to $2s$ of $v^\sigma$.  Then $d \leq 2a-2s$, and exactly $2a-2s-d$ entries from $2s+1$ to $n$ are equal to $\alpha^t$; so $\non 1_{T[2s+1,n]}(v^{\sigma}) = 2a-2s-d$.  When we right-multiply by $h_s$ (to obtain $\Rep(y)$), the entries of $v^{\sigma}$ from $1$ to $2s$ equal to $1_T$ become $\alpha$, and the remaining $d$ entries from $1$ to $2s$ (all equal to $\alpha^t$), all become $\alpha^t\alpha \neq \alpha$.  Hence $\nona_{[1,2s]}(\Rep(y)) = d$.  The entries $2s+1$ to $n$ of $v^\sigma$ are unchanged by $h_s$, so $\non 1_{T[2s+1,n]}(\Rep(y)) = \non 1_{T[2s+1,n]}(v^{\sigma}) = 2a-2s-d$.  It follows that $\nona_{[1,2s]}(\Rep(y)) + \non 1_{T[2s+1,n]}(\Rep(y)) = 2a-2s$.
\end{proof}

\begin{corollary} \label{SDdifficultorbitscor}
Let $x_i, x_j$ and $x_a$ be in $\Sigma$ with $1 \leq i < j < a \leq (n-1)/4$, and let $y \in x_a^{G_{x_i}} \cap x_a^{G_{x_j}}$.  Then entries $2i+1$ to $2j$ of $\Rep(y)$ are all $\alpha$.
\end{corollary}
\begin{proof}
By Lemma \ref{SDdagconditions} $\nona_{[1,2s]} (\Rep(y)) + \non 1_{T[2s+1,n]}(\Rep(y)) = 2a-2s$ for $s \in \{i,j\}$. Note that $\Rep(y) \in T^n$. Thus the conditions of Lemma \ref{difficultorbits2} hold with $\Delta=T$, $\beta=1_T$ and parameters $(i,j,n)$ being $(2i,2j,2a)$. The result now follows from Lemma \ref{difficultorbits2}.
\end{proof}

For an $n$-tuple $v$ we write $v_r$ to denote the $r^{\mathrm{th}}$-entry of $v$.

\begin{lemma} \label{SDincreasing1slem}
For $1 \leq b \leq (n-1)/4$, let $\Gamma_b$ be the pregeometry given by Construction {\em \ref{SDincreasing1s}}.  Then
\begin{itemize}
\item[(a)] $\Gamma_b$ is a geometry of rank $b$ and $G$ is flag-transitive on $\Gamma_b$,
\item[(b)] $\Gamma_b$ is thick if $b > 1$, and
\item[(c)] if $b>1$ then the rank $2$ truncations of $\Gamma_b$ are connected (and hence $\Gamma_b$ is a \ourgeoms{G}{\Omega}).
\end{itemize}
\end{lemma}
\begin{proof}
Part (a): The proof is by induction on $b$.  As $\ast$ comprises $(x,x)$ for $x \in X_1$, the statement is trivially true for $\Gamma_1$.  Suppose that $1 < b \leq \lfloor(n-1)/4\rfloor$ and that $\Gamma_{b-1}$ is a geometry with $G$ flag-transitive on $\Gamma_{b-1}$.  We check that the condition in Lemma \ref{balloonchambtrans}(b) holds for $\Gamma_b = \Inc(\Gamma_{b-1},G,K_{b-1},\Omega,x_{b},b)$. Let $Q$ be a non-empty subset of $K_{b-1}= \{x_1,\ldots,x_{b-1}\}$.  We need to show that $\bigcap_{x_s \in Q} G_{x_s}$ is transitive on $\bigcap_{x_s \in Q} x_b^{G_{x_s}}$.

Let $i$ be the smallest subscript in $Q$ and $j$ the largest.  If $i = j$ then $G_{x_i} = \bigcap_{x_s \in Q} G_{x_s}$ is transitive on $x_b^{G_{x_i}} = \bigcap_{x_s \in Q} x_b^{G_{x_s}}$.  So we may assume that $i < j$.  It follows from Corollary \ref{SDdifficultorbitscor} that:

\begin{eqnarray}
& & \text{all elements of $\RS(\bigcap_{x_s \in Q} x_b^{G_{x_s}})$ have entries $2i+1$ to $2j$ equal to $\alpha$.} \label{i+1toj}
\end{eqnarray}

Let $u$ and $u'$ be elements of $\bigcap_{x_s \in Q} x_b^{G_{x_s}}$.  Recall that for each $x_s \in Q$, $G_{x_s} = h_s^{-1} G_0 h_s$ where $G_0 = \langle D,S_n \rangle$ and $h_s$ is as in (\ref{hs2}) (just after Construction \ref{SDincreasing1s}).  Thus for each $y \in \bigcap_{x_s \in Q} x_b^{G_{x_s}}$ and $x_s \in Q$, there exists $\bar{t} \in D$ and $\sigma \in S_n$ such that $y = x_b^{h_s^{-1} \bar{t} \sigma h_s}$.

\nl
\textbf{Claim 1}: {\em Let $\bar{t}\sigma \in G_0$ be such that $u = x_b^{h_\ell^{-1} \bar{t}\sigma h_\ell}$ for some $x_\ell \in Q$, where $\bar{t} = (t,\ldots,t) \in D$ and $\sigma \in S_n$.  Then $t^{-1}\alpha t = \alpha$.}

\nl
Proof of Claim 1: Since $u \in x_b^{G_{x_i}}$ (where, as mentioned earlier, $i$ denotes the smallest subscript in $Q$) there exists $\bar{t}'\sigma' \in G_0$ such that $u = x_b^{h_i^{-1} \bar{t}'\sigma' h_i}$ with $\bar{t}'=(t',\ldots,t')\in D$ and $\sigma'\in S_n$.  Now $$\Rep(u) = \Rep(x_b^{h_i^{-1} \bar{t}'\sigma' h_i}) = (\underbrace{1_T,\ldots,1_T}_{2i},\underbrace{\alpha^{t'},\ldots,\alpha^{t'}}_{2a-2i},\underbrace{1_T,\ldots,1_T}_{n-2a})^{\sigma'}(\underbrace{\alpha,\ldots,\alpha}_{2i},\underbrace{1_T,\ldots,1_T}_{n-2i}).$$

By (\ref{i+1toj}), entries $2i+1$ to $2j$ of $\Rep(u)$ are equal to $\alpha$.  Since $\sigma'$ only permutes entries, and since right-multiplication by $(\underbrace{\alpha,\ldots,\alpha}_{2i},\underbrace{1_T,\ldots,1_T}_{n-2i})$ alters only the entries $1$ to $2i$, it follows from the above equation that the entries $2i+1$ to $2j$ of $\Rep(u)$ must all equal $\alpha^{t'}$, and this shows that $\alpha^{t'} = \alpha$.  Moreover, since $2a-2i > 2j-2i$ (and using the fact that $\alpha^{t'} = \alpha$), the vector $ (\underbrace{1_T,\ldots,1_T}_{2i},\underbrace{\alpha^{t'},\ldots,\alpha^{t'}}_{2a-2i},\underbrace{1_T,\ldots,1_T}_{n-2a})^{\sigma'}$ has at least one entry outside coordinates $2i+1$ to $2j$ which is equal to $\alpha$.  From this we get

\begin{eqnarray}
\label{abconds} & & \text{\parbox[t]{5in}{(a) There exists $r \in [1,n]$ such that either $r > 2j$ and $\Rep(u)_r = \alpha$, or $r \leq 2i$ and $\Rep(u)_r = \alpha^2$.}}\\
\nonumber & & \text{\parbox[t]{5in}{(b) If $r \leq 2i$ then $\Rep(u)_r \in \{\alpha,\alpha^2\}$ while if $r > 2j$ then $\Rep(u)_r \in \{1_T,\alpha\}$.}}  
\end{eqnarray}

If $\ell=i$, then the argument above with $\bar{t}'\sigma'=\bar{t}\sigma$ proves that $\alpha^t=\alpha$.  Suppose then that $\ell \neq i$, so $i < \ell \leq j$ (since $x_\ell \in Q$), and that $u = x_b^{h_\ell^{-1} \bar{t}\sigma h_\ell}$. Then $\Rep(u) = \Rep(x_b^{h_\ell^{-1} \bar{t}\sigma h_\ell}) = v.h_\ell$ where 
$$
v = (\underbrace{1_T,\ldots,1_T}_{2\ell},\underbrace{\alpha^{t},\ldots,\alpha^{t}}_{2a-2\ell},\underbrace{1_T,\ldots,1_T}_{n-2a})^{\sigma}\quad \mbox{and}\quad h_\ell = (\underbrace{\alpha,\ldots,\alpha}_{2\ell},\underbrace{1_T,\ldots,1_T}_{n-2\ell}).
$$

If $\ell < j$ then $\alpha = \Rep(u)_{2j} = v_{2j}.(h_\ell)_{2j} = v_{2j} \in \{1_T,\alpha^t\}$ and hence by (\ref{i+1toj}), $\alpha = v_{2j} = \alpha^t$ and Claim 1 is proved.  This leaves $\ell = j$.  Now for $r > 2j$, $\Rep(u)_r = v_r \in \{1_T,\alpha^t\}$, so if there exists $r > 2j$ with $\Rep(u)_r = \alpha$ then again we have $\alpha = \alpha^t$ proving Claim 1.  Suppose this is not the case.  Then by (\ref{abconds}), for some $r \leq 2i$, $\Rep(u)_r = \alpha^2$. However, we also have $\Rep(u)_r = v_r (h_\ell)_r = v_r \alpha \in \{\alpha,\alpha^t\alpha\}$. Since $\alpha \neq \alpha^2$ it follows that $\alpha^t\alpha = \alpha^2$, whence again $\alpha^t = \alpha$.

\nl
\textbf{Claim 2}: {\em There exists $z \in \bigcap_{x_s \in Q} G_{x_s}$ mapping $u$ to $u'$}.

\nl
Proof of Claim 2: Both $u$ and $u'$ are in $x_b^{G_{x_j}}$, and so $u = x_b^{h_j^{-1} {\bar{t}} \sigma h_j}$ and $u' = x_b^{h_j^{-1} {\bar{t}'} {\sigma'} h_j}$ for some $\bar{t} = (t,\ldots,t), \bar{t}' = (t',\ldots,t') \in D$ and $\sigma$, $\sigma' \in S_n$.

Recalling from Claim 1 that $\alpha^t = \alpha^{t'} = \alpha$, we have 
\begin{eqnarray} \label{Repuconds}
\text{(a)} & & \Rep(u) = \Rep(x_b^{h_j^{-1} \bar{t}\sigma h_j}) = v.h_j\\ 
\nonumber \text{(b)} & & \Rep(u') = \Rep(x_b^{h_j^{-1} \bar{t}'\sigma' h_j}) = v'.h_j
\end{eqnarray}
where $v = (\underbrace{1_T,\ldots,1_T}_{2j},\underbrace{\alpha,\ldots,\alpha}_{2a-2j},\underbrace{1_T,\ldots,1_T}_{n-2a})^{\sigma}$, $v' = (\underbrace{1_T,\ldots,1_T}_{2j},\underbrace{\alpha,\ldots,\alpha}_{2a-2j},\underbrace{1_T,\ldots,1_T}_{n-2a})^{\sigma'}$, and $h_j = (\underbrace{\alpha,\ldots,\alpha}_{2j},\underbrace{1_T,\ldots,1_T}_{n-2j})$.  Observe that $v$ and $v'$ both have the following properties:
\begin{itemize}
\item[(i)] entries $2i+1$ to $2j$ are all equal to $1_T$ (this follows from (\ref{i+1toj}) and (\ref{Repuconds}) and the fact that $\alpha^2=1_T$);
\item[(ii)] exactly $2a-2j$ entries are $\alpha$ and $n-2a+2j$ entries are $1_T$.
\end{itemize}

Hence there exists $\tau \in S_n < G_0$ such that $\tau$ fixes pointwise every coordinate from $2i+1$ to $2j$, and such that $\tau$ maps 
$$
(\underbrace{1_T,\ldots,1_T}_{2j},\underbrace{\alpha,\ldots,\alpha}_{2a-2j},\underbrace{1_T,\ldots,1_T}_{n-2a})^{\sigma}\quad \mbox{to}\quad (\underbrace{1_T,\ldots,1_T}_{2j},\underbrace{\alpha,\ldots,\alpha}_{2a-2j},\underbrace{1_T,\ldots,1_T}_{n-2a})^{\sigma'}.$$

Let $z$ be the element $h_j^{-1} \tau h_j$ of $G_{x_j}$.
Then 
\begin{eqnarray*}
\Rep(u^{z}) & = & \Rep(x_b^{h_j^{-1} \bar{t}\sigma h_jh_j^{-1} \tau h_j})= \Rep(x_b^{h_j^{-1} \bar{t}\sigma\tau h_j})\\
& = & (\underbrace{1_T,\ldots,1_T}_{2j},\underbrace{\alpha,\ldots,\alpha}_{2a-2j},\underbrace{1_T,\ldots,1_T}_{n-2a})^{\sigma\tau}(\underbrace{\alpha,\ldots,\alpha}_{2j},\underbrace{1_T,\ldots,1_T}_{n-2j})\\
& = & (\underbrace{1_T,\ldots,1_T}_{2j},\underbrace{\alpha,\ldots,\alpha}_{2a-2j},\underbrace{1_T,\ldots,1_T}_{n-2a})^{\sigma'}(\underbrace{\alpha,\ldots,\alpha}_{2j},\underbrace{1_T,\ldots,1_T}_{n-2j})\\
& = & \Rep(u').
\end{eqnarray*}
It follows from Lemma \ref{repdag} that $u^z = u'$.  Finally we prove that $z \in \bigcap_{x_s \in Q} G_{x_s}$.

Let $x_s \in Q$, and observe that since $i \leq s \leq j$ and $\tau$ fixes each coordinate from $2i+1$ to $2j$, the element $\tau$ fixes the vector $(\underbrace{1_T,\ldots,1_T}_{2s},\underbrace{\alpha^{-1},\ldots,\alpha^{-1}}_{2j-2s},\underbrace{1_T,\ldots,1_T}_{n-2j})$.  Thus we have the following. 
\begin{eqnarray*}
x_s^z & = & [(\underbrace{\alpha,\ldots,\alpha}_{2s},\underbrace{1_T,\ldots,1_T}_{n-2s})]^{h_j^{-1}\tau h_j}\\
& = & [(\underbrace{1_T,\ldots,1_T}_{2s},\underbrace{\alpha^{-1},\ldots,\alpha^{-1}}_{2j-2s},\underbrace{1_T,\ldots,1_T}_{n-2j})^{\tau}(\underbrace{\alpha,\ldots,\alpha}_{2j},\underbrace{1_T,\ldots,1_T}_{n-2j})]\\
& = & [(\underbrace{\alpha,\ldots,\alpha}_{2s},\underbrace{1_T,\ldots,1_T}_{n-2s})].
\end{eqnarray*}
Hence $z$ fixes all the $x_s$ in $Q$, and so $\bigcap_{x_s \in Q} G_{x_s}$ is transitive on $\bigcap_{x_s \in Q} x_b^{G_{x_s}}$.  Hence by Lemma \ref{balloonchambtrans}, $\Gamma_b$ is a geometry and $G$ is flag-transitive on $\Gamma_b$.  By induction, the result holds for each $b \leq (n-1)/4$.

\nl
Part (b):  Let $i$ be a type in $\Gamma_b$, so $1 \leq i \leq b$, and suppose that $b > 1$.  Let $K_b = \{x_1,\ldots,x_{b}\}$ be the chamber of $\Gamma_b$ given by Construction \ref{increasing1s}, and write $K' = K_b \backslash \{x_{i}\}$.  Then all elements $x_\ell$ of $K'$ are such that in $\Rep(x_\ell)$ entries $2i-1$ to $2i+2$ are either all equal to $\alpha$ or all equal to $1_T$.  Hence $G_{(K')}$ contains the subgroup $\Sym(\{2i-1,\ldots,2i+2\}) \leq S_n$ acting on the coordinates.  The orbit of
$$\Rep(x_{i}) = (\alpha,\ldots,\alpha,\underbrace{\alpha,\alpha,1_T,1_T,}_{\mathrm{entries} \; 2i-1 \; \mathrm{to} \; 2i+2}1_T,\ldots,1_T)$$
under the subgroup $\Sym(\{2i-1,\ldots,2i+2\})$ contains six distinct $n$-tuples, which (by Lemma \ref{uniquerep}) move $x_i$ to six distinct elements of $\Omega$.  Hence $|x_{i}^{G_{(K')}}| \geq 6$, and so the co-rank $1$ flag $K'$ is contained in at least 6 chambers.  Since $G$ is flag-transitive on $\Gamma_b$, it follows that $\Gamma_b$ is thick.

\nl
Part (c): Let $i,j$ be distinct types. As $G$ is primitive, we have $\langle G_{x_i}, G_{x_j} \rangle = G$, and it follows from Lemma \ref{balloonlem}(b) that the $\{i,j\}$-truncation is connected.
\end{proof}

\begin{remark}
\label{rem:SDdiag}
We do not determine the diagrams for the geometries yielded by  Construction \ref{SDincreasing1s}. Note that by \cite[Proposition 18.1]{priminc}, a primitive group with O'Nan-Scott type SD and whose minimal normal subgroup is isomorphic to $T^k$ with $k\geq 3$ cannot contain a primitive group of any other O'Nan-Scott type, and is not contained in a primitive group of another O'Nan-Scott type (other than the full alternating or symmetric group). Thus the geometries obtained are different to those obtained from Construction \ref{increasing1s} and Construction \ref{HSconstruction}.
\end{remark}

\section{Proof of Theorem \ref{maintheorem}} \label{mainthmproof}
As mentioned in the introduction we are concerned with the question of whether the rank of a thick flag-transitive geometry can be unboundedly large for a given O'Nan-Scott type of primitive group.   In this section we outline how the various constructions given in Sections \ref{genericconstr} to \ref{SDtype} enable us to prove Theorem \ref{maintheorem} for each of the O'Nan-Scott types.

\medskip\noindent
\textsc{Type AS}:  The projective space $\PG(k,q)$ is a thick geometry of rank $k$ upon which $G=\PGammaL(k+1,q)$ is flag-transitive. The group $G$ acts primitively of O'Nan-Scott type AS on each set of elements of a given type. However, for Theorem \ref{maintheorem} we require all the actions of our group $G$ to be permutationally isomorphic, which is not the case here.

Let $m = k+2$, and let $G = S_m$ acting on $\Omega = \{1,2,\ldots,m\}$.  Taking $b = m-2$, we use Construction \ref{ASconstruction} to construct a pregeometry $\Gamma$ of rank $b = k$.  By Lemma \ref{ASconstructionlem}, $\Gamma_b$ is a \ourgeoms{G}{\Omega} with $G^{\Omega}$ primitive of type AS.

\medskip\noindent
\textsc{Type PA}:
Theorem \ref{maintheorem} for the PA case follows from the following lemma.

\begin{lemma}
For each positive integer $k\geq 1$ and nonabelian simple group $T$, there exist a positive integer $n\geq 2$, a primitive group $G$ of O'Nan-Scott type \textrm{PA} with socle $T^n$ on a set $\Omega$ and a \ourgeoms{G}{\Omega} of rank $k$.
\end{lemma}
\begin{proof}
Let $n=2k+2$, let $T$ act primitively on a set $\Delta$ and let $G = T \Wr S_n$ act on $\Omega=\Delta^n$ in product action. As seen in Table \ref{PAexpansions}, $G$ is primitive of O'Nan-Scott type PA on $\Omega$ and $\soc(G)=T^n$. Taking $b = n/2-1=k$, we use Construction \ref{increasing1s} to construct a pregeometry $\Gamma$ of rank $b = k$.  By Lemma \ref{increasing1slem}, $\Gamma_b$ is a \ourgeoms{G}{\Omega}.
\end{proof}

Instead of varying the parameter $n$ to achieve arbitrary rank, we may also fix $n$ and vary $T$ to achieve the same outcome.

\begin{lemma}
Let $k\geq 1$ and $n\geq 2$ be positive integers. Then there exist a simple group $T$, a primitive group $G$ of O'Nan-Scott type \textrm{PA} with socle $T^n$ on a set $\Omega$ and a \ourgeoms{G}{\Omega} of rank $k$.
\end{lemma}
\begin{proof}
Let $m=k+2$, let $H=S_m$ and $\Delta=\{1,\ldots,m\}$. Taking $b=m-2=k$, Lemma \ref{ASconstructionlem} implies that Construction \ref{ASconstruction} yields a  \ourgeoms{H}{\Delta} $\Gamma_b$ of rank $k$ with $H^{\Delta}$ primitive of type AS. By Lemma \ref{lem:productup}, we can then apply Construction \ref{con:prod1} to $\Gamma_b$ to obtain a \ourgeoms{G}{\Omega} $\Gamma=\Sigma^n$ of rank $k$, where $G=H\Wr S_n$ acts primitively of type PA on  $\Omega=\Delta^n$. The socle of $G$ is $A_m^n$.
\end{proof}

\medskip\noindent
\textsc{Type HS}: Let $m=4k+1$, $T=A_m$ and $G=T\times T$ acting on $\Omega=T$ as in Construction \ref{HSconstruction}. Letting $b=\lfloor (n-1)/4\rfloor$,  Lemma \ref{HSconstructionlem} implies that the geometry $\Gamma_b$ yielded by Construction \ref{HSconstruction} is a \ourgeoms{G}{\Omega}  of rank $k$. Moreover, $G^{\Omega}$ is primitive of type HS.

\medskip\noindent
\textsc{Type HC}: 
Theorem \ref{maintheorem} for the HC case follows from the following lemma.

\begin{lemma}
For each positive integer $k\geq 1$ and nonabelian simple group $T$, there exist a positive integer $n\geq 1$, a primitive group $G$ of O'Nan-Scott type \textrm{HC} with socle $T^{2n}$ on a set $\Omega$ and a \ourgeoms{G}{\Omega} of rank $k$.
\end{lemma}
\begin{proof}
Let $n=2k+2$, let $H=T\times T$ act primitively of O'Nan-Scott type HS on a set $\Delta=T$, and let $G = H \Wr S_n$ in product action on $\Omega=\Delta^n$. As seen in Table \ref{PAexpansions}, $G$ is primitive of O'Nan-Scott type HC on $\Omega$ and $\soc(G)=T^{2n}$. Taking $b = n/2-1=k$, we use Construction \ref{increasing1s} to construct a pregeometry $\Gamma_b$ of rank $b = k$.  By Lemma \ref{increasing1slem}, $\Gamma_b$ is a \ourgeoms{G}{\Omega}.
\end{proof}

Instead of varying the parameter $n$ to achieve arbitrary rank, we may also fix $n$ and vary $T$ to achieve the same outcome.

\begin{lemma}
Let $k,n\geq 1$ be positive integers. Then there exist a simple group $T$, a primitive group $G$ of O'Nan-Scott type \textrm{HC} with socle $T^{2n}$ on a set $\Omega$ and a \ourgeoms{G}{\Omega} of rank $k$.
\end{lemma}
\begin{proof}
Let $m=4k+1$ and let $H=A_m\times A_m$ acting primitively of type HS on $\Delta=A_m$. As seen in the discussion of the HS case we can construct a \ourgeoms{H}{\Delta} $\Sigma$ of rank $k$ from Construction \ref{HSconstruction}. For each positive integer $n\geq 1$ we can then apply Construction \ref{con:prod1} to $\Sigma$ to obtain a pregeometry $\Gamma=\Sigma^n$ of rank $k$, where $G=H\Wr S_n$ acts primitively of type HC on $\Omega=\Delta^n$. The socle of $G$ is $A_m^{2n}$. By Lemma \ref{lem:productup}, $\Gamma$ is a \ourgeoms{G}{\Omega}. 
\end{proof}

\medskip\noindent
\textsc{Type SD}:  
Using the notation of Remark \ref{rem:HStoSD}, the \ourgeoms{G}{\Omega} of rank $k$ given by Construction \ref{HSconstruction} with $G$ primitive of type HS on $\Omega$ is also an \ourgeoms{H}{\Omega} with $H$ primitive of type SD on $\Omega$. This example is sufficient to prove Theorem \ref{maintheorem} in the SD case. It has $\soc(H)=A_m\times A_m$ with two simple direct factors. However, it is possible to have a primitive group of type SD whose socle has an arbitrary number of simple direct factors and where these factors are not alternating groups. Such groups can still give rise to a \ourgeoms{G}{\Omega} of arbitrarily large rank.

\begin{lemma}
\label{lem:SDvarynT}
Let $n\geq 2$, $T$ be a nonabelian simple group and let $k=\max\{2,\lfloor(n-1)/4\rfloor\}$. Then there exists a primitive group $G$ of type \textrm{SD} with socle $T^n$ on a set $\Omega$ and a \ourgeoms{G}{\Omega} of rank $k$.
\end{lemma}
\begin{proof}
Let $G$ be the primitive group of type SD on the set $\Omega$ used in Construction \ref{SDincreasing1s}.  The socle of $G$ is $T^n$.

Suppose first that $k=2$. Let $\Sigma=\{x_1,x_2\}\subset \Omega$ with $x_1\neq x_2$. Let $X_1 = \Omega$, let $\Gamma_1 = (X_1,\ast,t)$, where $\ast$ consists of the pairs $(x,x)$ for $x \in X_1$, and $t(x) = 1$ for all $x \in X_1$. Let $K_1 = \{x_1\}$, and let $\Gamma=\mathrm{Inc}(\Gamma_1,G,K_1,\Omega,x_2,2)$ as given by Construction \ref{balloon}. Then by Lemma \ref{balloonlem}, $\Gamma$ is a \ourgeoms{G}{\Omega} of rank 2.

Next suppose that $(n-1)/4\geq 3$ and let $b=k=\lfloor(n-1)/4\rfloor$. By Lemma  \ref{SDincreasing1slem} the geometry $\Gamma_b$ yielded by Construction \ref{SDincreasing1s} is a \ourgeoms{G}{\Omega} of rank $k$.
\end{proof}

The following question remains open.
\begin{question}
\textrm{
Let $n,k\geq 2$ be positive integers. Is there a nonabelian simple group $T$, a primitive group $G$ of type \textrm{SD} on a set $\Omega$ with $\soc(G)=T^n$ and a \ourgeoms{G}{\Omega} of rank $k$? }
\end{question}

Note that taking truncations of the examples produced to prove Lemma~\ref{lem:SDvarynT} gives a positive answer if $k\leq\lfloor(n-1)/4\rfloor$, while Remark~\ref{rem:HStoSD} provides an example for all $k$ if $n=2$.

\medskip\noindent
\textsc{Type CD}:
Theorem \ref{maintheorem} for the CD case follows from the next lemma.

\begin{lemma}
Let $k\geq 1$ and $n\geq 2$. Then there exist a nonabelian simple group $T$, a primitive group $G$ of type CD with socle $T^{2n}$ on a set $\Omega$ with $(T^{2n})_{\alpha}\cong T^n$ and a \ourgeoms{G}{\Omega} $\Gamma$ of rank $k$.
\end{lemma}
\begin{proof}
Let $m=4k+1$, $T=A_m$ and $H=T\wr S_2$ acting on $\Delta=T$ as a primitive group of type SD as in Remark \ref{rem:HStoSD}. Letting $b=(m-1)/4$,  Lemma \ref{HSconstructionlem}  and Remark \ref{rem:HStoSD} imply that the geometry $\Sigma=\Gamma_b$ yielded by Construction \ref{HSconstruction} is an \ourgeoms{H}{\Delta} of rank $k$. 

For each positive integer $n\geq 2$ we can then apply Construction \ref{con:prod1} to $\Sigma$ to obtain a \ourgeoms{G}{\Omega} $\Gamma=\Sigma^n$ of rank $k$, where $G=H\Wr S_n$ acts primitively of type CD on $\Omega=\Delta^n$. The socle of $G$ is $T^{2n}$ and $(T^{2n})_{\alpha}\cong T^n$.
\end{proof}

As for the previous O'Nan-Scott types there is some flexibility in how we can achieve arbitrary rank.

\begin{lemma}
Let $k\geq 1, r\geq 2$ be positive integers and let $T$ be a finite nonabelian simple group. Then there exist a positive integer $n$, primitive group $G$ of type \textrm{CD} with socle $T^{rn}$ on a set $\Omega$ with $(T^{rn})_{\alpha}\cong T^n$ and a \ourgeoms{G}{\Omega} of rank $k$.
\end{lemma}
\begin{proof}
Let $n = 2k + 2$, let $H$ be a primitive group on $\Delta$ of type SD with socle $T^r$ and let $G = H \Wr S_n$ in product action on $\Omega=\Delta^n$.  As seen in Table \ref{PAexpansions}, $G$ is primitive of O'Nan-Scott type CD on $\Omega$. Taking $b = n/2-1=k$, we use Construction \ref{increasing1s} to construct a pregeometry $\Gamma_b$ of rank $b = k$.  By Lemma \ref{increasing1slem}, $\Gamma_b$ is a \ourgeoms{G}{\Omega} with $G^{\Omega}$ primitive of type CD. The socle of $G$ is $T^{rn}$ and $(T^{rn})_\alpha\cong T^n$.
\end{proof}

\begin{lemma}
\label{lem:CDvarynrT}
Let $n,r\geq 2$ be positive integers, $T$ be a nonabelian simple group and let $k=\max\{2,\lfloor(r-1)/4\rfloor\}$. Then there exist a primitive group $G$ of type CD with socle $T^{rn}$ on a set $\Omega$ with $(T^{rn})_{\alpha}\cong T^n$ and a \ourgeoms{G}{\Omega} $\Gamma$ of rank $k$.
\end{lemma}
\begin{proof}
By Lemma \ref{lem:SDvarynT}, there exists a \ourgeoms{H}{\Delta} $\Sigma$ of rank $k$ with $H$ a primitive group of type SD on $\Delta$ and $\soc(H)=T^r$. Then by Lemma \ref{lem:productup}, we can  apply Construction \ref{con:prod1} to $\Sigma$ to form a new geometry $\Gamma=\Sigma^n$ which is a \ourgeoms{G}{\Omega} of rank $k$ where $G=H\Wr S_n$ is primitive of type CD on $\Omega=\Delta^n$. 
\end{proof}

The following question remains open.

\begin{question}
\textrm{
For given integers $n,k,r$, all at least $2$, is there a nonabelian simple group $T$, a primitive group $G$ of type \textrm{CD} on a set $\Omega$ with $\soc(G)=T^{rn}$ and $\soc(G)_{\alpha}\cong T^n$, and a  \ourgeoms{G}{\Omega} of rank $k$? }
\end{question}

Note that taking truncations of the examples produced to prove Lemma~\ref{lem:CDvarynrT} gives a positive answer if $k\leq\lfloor(r-1)/4\rfloor$.

\noindent
\textsc{TW and HA}:  Let $n = 2k + 2$, let $H$ be a primitive group on $\Delta$ of O'Nan-Scott type TW or HA and let $G = H \Wr S_n$ in product action on $\Delta^n$.  As seen in Table \ref{PAexpansions}, $G$ is primitive of O'Nan-Scott type TW or HA respectively on $\Omega$. Taking $b = n/2-1=k$, we use Construction \ref{increasing1s} to construct a pregeometry $\Gamma_b$ of rank $b = k$.  By Lemma \ref{increasing1slem}, $\Gamma_b$ is a \ourgeoms{G}{\Omega} with $G^{\Omega}$ primitive of type TW or HA. This completes the proof of Theorem \ref{maintheorem}.

For each nonabelian simple group $T$ and positive integer $n\geq 2$ there is not necessarily a primitive permutation group $G$  of type TW with $\soc(G)=T^n$. For example for $T=A_5$ the smallest possible value of $n$ is 6. We note that by \cite[Proposition 9.4]{badTW}, for each nonabelian simple group $T$, there is a primitive permutation group $G$ of type TW with socle $T^{|T|}$.

For the HA case it could be asked whether it is possible to achieve arbitrary rank by varying the prime $p$ instead of the dimension of the vector space. However, the following lemma shows that this is not possible.

\begin{lemma}
Let $G\leqslant \AGL(d,p)$ act primitively of type \textrm{HA} on $\Omega$ and suppose that $\Gamma$ is a \ourgeoms{G}{\Omega} of rank $k$. Then $k\leq d+1$.
\end{lemma}
\begin{proof}
Let $K=\{x_1,x_2,\ldots,x_k\}$ be a chamber of $\Gamma$. Since $G$ is thick and chamber-transitive we have  $G_{x_1,x_2,\ldots,x_k}<G_{x_1,x_2,\ldots,x_{k-1}}<\cdots<G_{x_1}$. Since $G$ acts transitively on $\Omega=\GF(p)^d$ we may assume that $x_1$ is the zero vector and so $G_{x_1}=G\cap \GL(d,p)$. Then choosing $x_2$ as our first basis vector we have that $G_{x_1,x_2}$ consists of all matrices in $G_{x_1}$ of the form
$$\begin{pmatrix}
   1& 0 \\
   v^{T}&A
  \end{pmatrix}$$
where $A\in\GL(d-1,p)$ and $v\in\GF(p)^{d-1}$. Since $G_{x_1,x_2}\neq G_{x_1,x_2.x_3}$ it follows that $x_3\in \Omega\backslash\langle x_2\rangle$. Hence we can choose $x_3$ as our second basis vector and so $G_{x_1,x_2.x_3}$ consists of all matrices in $G_{x_1,x_2}$ of the form
$$\begin{pmatrix}
   1& 0 &0\\
   0&1&0\\
   v^{T}&w^T&A
  \end{pmatrix}$$
where $A\in\GL(d-2,p)$ and $v,w\in\GF(p)^{d-2}$. Proceeding in this fashion we see that for each $i$ we have $x_i\in \Omega\backslash\langle x_1,\ldots,x_{i-1}\rangle$ and if $k\geq d+1$ then $G_{x_1,x_2,\ldots,x_{d+1}}=1$. Hence $k\leq d+1$.
\end{proof}

\section{Combining different primitive actions}\label{differenttypes}

We note that in our constructions so far, $G$ has had the same action on every set $X_i$.  However, it is possible for a given permutation group to have different O'Nan-Scott types of faithful primitive (or quasiprimitive) actions on different sets. For example, let $G=A_5\Wr A_6$. Then $G$ acts primitively with O'Nan-Scott type PA on a set $X_1$ of size $5^6$, with O'Nan-Scott type SD on a set $X_2$ of size $60^5$ and with O'Nan-Scott type TW on a set $X_3$ of size $60^6$. Magma \cite{magma} calculations show that it is possible to construct a thick flag-transitive geometry of rank 3 for $G$ whose set of elements is $X_1\cup X_2\cup X_3$ and whose rank 2 truncations are connected.

\begin{question}
Suppose that $G$ is flag-transitive on a thick geometry $\Gamma$ with connected rank 2 truncations such that for each type $i$, $G^{X_i}\cong G$ is primitive.  Suppose also that there exist distinct types $i,j$ such that $G^{X_i}$ and $G^{X_j}$ have primitive actions of different kinds.  Is there a bound on the rank of $\Gamma$?
\end{question}


\begin{thebibliography}{10}

\bibitem{badTW}
R.~W. Baddeley.
\newblock Primitive permutation groups with a regular nonabelian normal
  subgroup.
\newblock {\em Proc. London Math. Soc. (3)}, 67(3):547--595, 1993.

\bibitem{magma}
W.~Bosma, J.~Cannon, and C.~Playoust.
\newblock The {M}agma algebra system {I}: The user language.
\newblock {\em J. Symb. Comp.}, 24 3/4:235--265, 1997.
\newblock Also see the Magma home page at
  http://www.maths.usyd.edu.au:8000/u/magma/.

\bibitem{IncGeomHB}
F.~Buekenhout, editor.
\newblock {\em Handbook of incidence geometry}.
\newblock North-Holland, Amsterdam, 1995.
\newblock Buildings and foundations.

\bibitem{bueetal}
F.~Buekenhout, P.~Cara, M.~Dehon and D.~Leemans, Residually weakly primitive geometries of small sporadic and almost simple groups: a synthesis.  \emph{Topics in diagram geometry},  1–27, Quad. Mat., 12, Dept. Math., Seconda Univ. Napoli, Caserta, 2003.

\bibitem{CameronBook}
P.~J. Cameron.
\newblock {\em Permutation groups}, volume~45 of {\em London Mathematical
  Society Student Texts}.
\newblock Cambridge University Press, Cambridge, 1999.

\bibitem{CGP}
P.~Cara, A.~Devillers, M.~Giudici, and C.~E. Praeger.
\newblock Quotients of incidence geometries.
\newblock to appear in \emph{Des. Codes Cryptogr.}

\bibitem{deho94}
M.~Dehon.
\newblock Classifying geometries with {CAYLEY}.
\newblock {\em J. Symbolic Comput.}, 17(3):259--276, 1994.

\bibitem{DM}
J.~D. Dixon and B.Mortimer.
\newblock {\em Permutation groups}, volume 163 of {\em Graduate Texts in
  Mathematics}.
\newblock Springer-Verlag, New York, 1996.

\bibitem{firstpaper}
M.~Giudici, C.~H.~Li, G.~Pearce and C.~E. Praeger.
\newblock Basic and degenerate pregeometries.
\newblock submitted, arXiv:1009.0075v1.

\bibitem{GLP3}
M.~Giudici, C.~H.~Li, and C.~E. Praeger.
\newblock Analysing finite locally {$s$}-arc transitive graphs.
\newblock {\em Trans. Amer. Math. Soc.}, 356(1):291--317, 2004.

\bibitem{psl2}
J.~De Saedeleer and D.~Leemans, On the rank two geometries of the groups ${\rm PSL}(2,q)$: part I.  \emph{Ars Math. Contemp.}  3(2): 177--192, 2010.


\bibitem{Leemans}
D.~Leemans.  Residually weakly primitive and locally two-transitive geometries for sporadic groups. Acad. roy. Belgique, Mem. Cl. Sci., Coll. 4, Ser. 3, Tome XI(2008), 173pp.

\bibitem{leemanssuz}
D.~Leemans,  The residually weakly primitive pre-geometries of the Suzuki simple groups.  \emph{Note Mat.}  20(1):1--20, 2000/01.

\bibitem{priminc}
C.~E.~Praeger, The inclusion problem for finite primitive permutation groups.  \emph{Proc. London Math. Soc. (3)}  60(1):68--88, 1990.

\bibitem{PraegerQP}
C.~E. Praeger.
\newblock An {O}'{N}an-{S}cott theorem for finite quasiprimitive permutation
  groups and an application to {$2$}-arc transitive graphs.
\newblock {\em J. London Math. Soc. (2)}, 47(2):227--239, 1993.

\bibitem{praegerbcc}
C.~E. Praeger.
\newblock Finite quasiprimitive graphs.
\newblock In {\em Surveys in combinatorics, 1997 ({L}ondon)}, volume 241 of
  {\em London Math. Soc. Lecture Note Ser.}, pages 65--85. Cambridge Univ.
  Press, Cambridge, 1997.

\bibitem{RS}
M.~A.~Ronan and D.~Stroth, Minimal parabolic geometries for the sporadic groups.
\emph{European J. Combin.} 5(1): 59--91, 1984.

\bibitem{SmithDH}
D.~H. Smith.
\newblock Primitive and imprimitive graphs.
\newblock {\em Quart. J. Math. Oxford Ser. (2)}, 22:551--557, 1971.

\bibitem{Tits}
J.~Tits. 
\newblock G\'eom\'etries poly\'edriques et groupes simples.
\newblock in \emph{Atti 2a Riunione Groupem. Math. Express. Lat. Firenze} (1962), 66--88.

\end{thebibliography}
\end{document}